\documentclass[journal]{IEEEtran}
\renewcommand{\baselinestretch}{1.0}
\renewcommand{\theenumii}{\theenumi.\arabic{enumii}}

\usepackage[utf8]{inputenc}
\usepackage{color,soul}
\usepackage[T1]{fontenc}
\usepackage{multicol}
\usepackage{multirow}
\usepackage{textcomp}
\usepackage{bigdelim}
\usepackage{graphicx}
\usepackage{epstopdf}
\usepackage{amssymb}
\usepackage{amsmath}
\usepackage{psfrag}
\usepackage{amsthm}
\usepackage[tight,footnotesize]{subfigure}
\usepackage{cite}
\usepackage{enumerate}
\usepackage{algorithm}
\usepackage{algorithmic}
\usepackage{mathtools}
\usepackage{mathabx}
\usepackage{mathrsfs} 
\usepackage{array}
\usepackage{makecell}
\theoremstyle{definition}
\newtheorem{assumption}{Assumption}
\newtheorem{theorem}{Theorem}

\newtheorem{lemma}{Lemma}

\newtheorem{remark}{Remark}

\makeatletter
\usepackage{etoolbox}
\makeatletter 
\pretocmd\@bibitem{\color{black}\csname keycolor#1\endcsname}{}{\fail}
\newcommand\citecolor[1]{\@namedef{keycolor#1}{\color{blue}}}
\makeatother

\newcommand{\vast}{\bBigg@{3.2}}
\newcommand{\Vast}{\bBigg@{4.5}}

\makeatother
\graphicspath{figs/}

\begin{document}
	
\title{Improved Hierarchical ADMM for Nonconvex Cooperative Distributed Model Predictive Control}

\author{\IEEEauthorblockN{Xiaoxue Zhang, Jun Ma, Zilong Cheng, Sunan Huang, Clarence W. de Silva, \IEEEmembership{Fellow,~IEEE,}
		and
		Tong Heng Lee}
	\thanks{X. Zhang, Z. Cheng, and T. H. Lee are with the Integrative Sciences and Engineering Programme, NUS Graduate School, National University of Singapore, 119077
		(e-mail: xiaoxuezhang@u.nus.edu; zilongcheng@u.nus.edu; eleleeth@nus.edu.sg).}
	\thanks{J. Ma is with the Department of Mechanical Engineering, University of California, Berkeley, CA 94720 USA
		(e-mail: jun.ma@berkeley.edu).}
\thanks{S. Huang is with the Temasek Laboratories, National University of Singapore, Singapore, 117411 (e-mail: tslhs@nus.edu.sg).}	
\thanks{Clarence W. de Silva is with the Department of Mechanical Engineering, University of British Columbia, Vancouver, BC, Canada V6T 1Z4 (e-mail: desilva@mech.ubc.ca).}
\thanks{This work has been submitted to the IEEE for possible publication. Copyright may be transferred without notice, after which this version may no longer be accessible.}
}
	
	\markboth{}
	{X. Zhang \MakeLowercase{\textit{et al.}}}
	
	\maketitle
	
\begin{abstract}
Distributed optimization is often widely attempted and innovated as an attractive and preferred methodology to solve large-scale problems effectively in a localized and coordinated manner. Thus, it is noteworthy that the methodology of distributed model predictive control (DMPC) has become a promising approach to achieve effective outcomes, e.g., in decision-making tasks for multi-agent systems. However, the typical deployment of such distributed MPC frameworks would lead to the involvement of nonlinear processes with a large number of nonconvex constraints. To address this important problem, the development and innovation of a hierarchical three-block alternating direction method of multipliers (ADMM) approach is presented in this work to solve this nonconvex cooperative DMPC problem in multi-agent systems. Here firstly, an additional slack variable is introduced to transform the original large-scale nonconvex optimization problem. Then, a hierarchical ADMM approach, which contains outer loop iteration by the augmented Lagrangian method (ALM) and inner loop iteration by three-block semi-proximal ADMM, is utilized to solve the resulting transformed nonconvex optimization problem. Additionally, it is analytically shown and established that the requisite desired stationary point exists for convergence in the algorithm. Finally, an approximate optimization stage with a barrier method is then applied to further significantly improve the computational efficiency, yielding the final improved hierarchical ADMM. 
The effectiveness of the proposed method in terms of attained performance and computational efficiency is demonstrated on
a cooperative DMPC problem of decision-making process for multiple unmanned aerial vehicles (UAVs).
\end{abstract}

\begin{IEEEkeywords}
Alternating direction method of multipliers (ADMM), distributed optimization, model predictive control (MPC), multi-agent system, collision avoidance.
\end{IEEEkeywords}

\maketitle

\IEEEdisplaynontitleabstractindextext

\IEEEpeerreviewmaketitle

\section{Introduction}
With the rapid development of communication and computation technologies, cooperative decision making of large-scale multi-agent systems has become a promising trend in the development of automated systems~\cite{he2017continuous, hong2016distributed,ye2020global}. Nevertheless, the significant challenge involved in cooperative decision making, in handling the coupling conditions among the interconnected agents with an appropriate guarantee of efficient computation still remains. In the existing literature, a significant amount of effort has been put into distributed control and coordinated networks for large-scale interconnected multi-agent systems, with the objective of optimizing the global cost while satisfying the required safety constraints collectively. Essentially for such approaches, the control and optimization algorithms deployed are distributed in structure, and the deployment utilizes only the local observations and information~\cite{ma2019parameter,yang2016distributed}. With the appropriate communication technology in place, the agents are developed to suitably make decisions while attempting to satisfy all of the constraints or requirements invoked~\cite{yang2016multi, yuan2011distributed}. On the other hand, a remarkable alternate development that utilizes an approach termed as distributed model predictive control (DMPC) has also been shown to display great promise to be effectively used to handle the input and state constraints with multiple objectives; and possibly also applicable to this important problem of cooperative decision making of large-scale multi-agent systems~\cite{camponogara2009distributed}. In addition, via the DMPC framework, there is great potential that the computational burden can be significantly relieved by decomposing the centralized system into several interconnected subsystems (with also a similar situation of availability of knowledge of necessary information among the connected agents).

For the DMPC problem with interconnected agents, the promising optimization approach of the alternating direction method of multipliers (ADMM) algorithm has the potential for very effective utilization. In this context, it can be noted that a large number of research works have demonstrated that the ADMM methodology has the capability to rather effectively determine the optimal solution to many challenging optimization problems, such as distributed optimization~\cite{hong2017distributed}, decentralized control~\cite{ma2020symmetric, ma2020optimal}, and statistical learning problems~\cite{hu2016admm, ding2019optimal}. The key idea of the ADMM approach is to utilize a decomposition coordination principle, wherein local sub-problems with smaller sizes are solved, and then the solutions of these subproblems can be coordinated to obtain the solution to the original problem. Notably, the ADMM solves the sub-problems by optimizing blocks of variables, such that efficient distributed computation can be attained. The convergence of the ADMM with two blocks has been investigated for convex optimization problems in~\cite{he20121, monteiro2013iteration}, and typical applications of such two-block ADMM approaches in distributed convex optimization problems are reported extensively in~\cite{wei2012distributed, makhdoumi2017convergence}. Several researchers have also made rather noteworthy extensions from this two-block convex ADMM to multi-block convex ADMM in~\cite{he2012alternating, lin2015global}.
However, although seemingly evident as a possibility with great potential, there has not been much substantial work on innovating the ADMM framework for nonconvex DMPC. Classical ADMM is known only to admit a linear convergence rate for convex optimization problems. 

For nonconvex optimization problems, certain additional conditions and assumptions need to be considered to guarantee the required convergence~\cite{li2015global, wang2014bregman, jiang2019structured, wang2019global,wang2019distributed,cheng2021admm,ma2020alternating}. More specifically,
the work in~\cite{li2015global} studies the two-block ADMM with one of the blocks defined as the identity matrix; and it also states that the objective function and the domain of the optimization variable need to be lower semi-continuous. Yet further, a nonconvex Bregman-ADMM is proposed in~\cite{wang2014bregman} by introducing a Bregman divergence term to the augmented Lagrangian function to promote the descent of certain potential functions. 
Moreover, the work in~\cite{jiang2019structured} proposes two proximal-type variants of ADMM with $\epsilon$-stationarity conditions to solve the nonconvex optimization problem with affine coupling constraints under the assumption that the gradient of the objective function is Lipschitz continuous. 
Additionally, in \cite{wang2019global}, a multi-block ADMM is presented to minimize a nonconvex and nonsmooth objective function subject to certain coupling linear equality constraints; and the work in~\cite{wang2019distributed} presents an ADMM-type algorithm for nonconvex optimization problem in the application environment of modern machine learning problems, though requiring the assumption that the penalty parameter should be large enough. Besides, an augmented Lagrangian based alternating direction inexact newton (ALADIN) method is introduced in~\cite{houska2016augmented} to solve the optimization problems with a separable nonconvex objective function and coupled affine constraints.

Noting all these additional conditions and assumptions that typically can arise, it is the situation where the ADMM approach cannot yet be directly applied to solve a large-scale constrained nonconvex program without further relaxation or reformulation. Also, the ADMM with more than two blocks is typically implemented with sequential optimization; and while a direct parallelization of a multi-block ADMM formulation is also sometimes considered, the aforementioned approach has the disadvantage that it is typically less likely to converge~\cite{sun2019two,tang2020fast}. In order to address all these limitations, a hierarchical three-block ADMM approach is utilized in our paper to solve the nonconvex optimization problem that arises for decision making in multi-agent systems. Firstly, an additional slack variable is innovatively introduced to transform the original large-scale nonconvex optimization problem, such that the intractable nonconvex coupling constraints are suitably related to the distributed agents. Then, the approach with a hierarchical ADMM that contains the outer loop iteration by the augmented Lagrangian method (ALM), and inner loop iteration by three-block semi-proximal ADMM, is utilized to solve the resulting transformed nonconvex optimization problem. Next, the approximate optimization with a barrier method is then applied to improve the computational efficiency. Finally, a multi-agent system involving decision making for multiple unmanned aerial vehicles (UAVs) is utilized to demonstrate the effectiveness of the proposed method in terms of attained performance and computational efficiency.

The remainder of this paper is organized as follows. Section~\ref{section:problem_formulation} describes a multi-agent system that contains multiple interconnected agents equipped with sensors and communication devices, and also formulates a large-scale nonconvex optimization problem for decision making. Section~\ref{section:problem_reformulation} gives a more compact form of the optimization problem, which is further transformed by introducing slack variables. In section~\ref{section:two_level_ADMM}, the hierarchical ADMM  is used to solve the transformed optimization problem. Next, an improved version of such hierarchical ADMM is presented to accelerate the computation process based on the barrier method in Section~\ref{section:improved_Hierarchical_ADMM}. Then, in Section~\ref{section:Illustrative_Examples}, a multi-UAV system is used as an example to show the effectiveness and computational efficiency of the hierarchical ADMM and improved hierarchical ADMM. Finally, pertinent conclusions of this work are given in Section~\ref{section:conclusion_and_discussion}.

\section{Problem Formulation}
\label{section:problem_formulation}
The following notations are used in the remaining content. $\mathbb R^{a\times b}$ denotes the set of real matrices with $a$ rows and $b$ columns, $\mathbb R^{a}$ means the set of $a$-dimensional real column vectors. The symbol $X\succ 0$ and $X\succeq 0$ mean that the matrix $X$ is positive definite and positive semi-definite, respectively. $x>y$ and $x\geq y$ mean that vector $x$ is element-wisely greater and no less than the vector $y$. $A^{\top}$ and $x^{\top}$ denote the transpose of the matrix $A$ and vector $x$. $I_{a}$ represents the $a$-dimensional identity matrix; $1_{a}$ and $1_{(a,b)}$ denote the $a$-dimensional all-one vector and the $a$-by-$b$ all-one matrix, respectively; $0_{a}$ and $0_{(a,b)}$ represent the $a$-dimensional all-zero vector and the $a$-by-$b$ all-zero matrix, respectively.  The operator $\langle X, Y\rangle$ denotes the Frobenius inner product, i.e., $\langle X, Y\rangle = \operatorname{Tr}(X^{\top}Y)$ for all $X,Y\in \mathbb R^{a\times b}$; the operator $\|X\|$ is the Frobenius norm of matrix $X$. $\otimes$ denotes the Kronecker product and $\odot$ denotes the Hadamard product (Schur product). $\operatorname{blockdiag}(X_1, X_2, \cdots, X_n)$ denotes a block diagonal matrix with diagonal entries $X_1, X_2, \cdots, X_n$; $\operatorname{diag}(a_1, a_2, \cdots, a_n)$ denotes a diagonal matrix with diagonal entries $a_1, a_2, \cdots, a_n$. $\mathbb Z^{a}$ and $\mathbb Z_{a}^{b}$ mean the sets of positive integers $\{1,2,\cdots,a\}$ and $\{a,a+1,\cdots,b\}$, respectively. $\mathbb Z_{+}$ denotes the set of positive integers $\{1,2,\cdots\}$. The operator $\{x_i\}_{\forall i \in \mathbb Z_a^b}$ means the concatenation of the vector $x_i$ for all $i\in \mathbb Z_a^b$, i.e., $\{x_i\}_{\forall i \in \mathbb Z_a^b} = (x_a, x_{a+1}, \cdots, x_b) = \begin{bmatrix} x_a^\top & x_{a+1}^\top & \cdots & x_b^\top \end{bmatrix}^\top$.

\subsection{Modelling of Interconnected Agents}
In order to represent the constraint or information topology of multiple interconnected agents, an undirected graph $\mathcal G(\mathcal V, \mathcal E)$ can be utilized. The node set $\mathcal V=\{1,2, \cdots, N\}$ denotes the agents, and the edge set $\mathcal E$ denotes the coupling constraints (information flow) between two interconnected agents, which is defined as 
\begin{equation}
\label{eq:comunication_top}
\begin{cases}
(i,j) \in \mathcal E(t), & \; d_{\mathrm{safe}} \leq \|p_i-p_j\|\leq d_{\mathrm{cmu}}\\
(i,j) \notin \mathcal E(t), & \; \text{otherwise},
\end{cases}
\end{equation}
where $N$ and $M$ are the number of agents and coupling constraints in the multi-agent system, respectively, $p_i, p_j$ are the position vectors of the $i$th agent and $j$th agent, $d_{\mathrm{cmu}}$ and $d_{\mathrm{safe}}$ mean the maximum communication distance and minimum safe distance between two agents, respectively. Therefore, based on the communication topology, an adjacency matrix of the network (denoted by $\mathcal D$) can be obtained, which is a square symmetric matrix. The elements of $\mathcal D$ indicate whether pairs of vertices are adjacent/connected or not in the graph. Therefore, the neighbor nodes of the $i$th agent are the corresponding indexes of nonzero elements of the $i$th row in $\mathcal D$, which is denoted by $\nu (i) = \{j|(i,j)\in \mathcal E, \forall j\in \mathcal V\}$.

\subsection{Problem Description}
Each connected agent in the set $\mathcal V$ has its origin and target state, which means each agent needs to achieve its task by coordinating with other connected agents in $\mathcal V$. Moreover, the information of the connected agents (neighbors) is necessary for each agent to make decisions, due to the requirement of the communication topology~\eqref{eq:comunication_top}. Since the communication between connected agents occurs during the whole process, the present state information of neighbors needs to be conveyed to the $i$th agent to make decisions. Thus, the cooperative task can be formulated as a cooperative DMPC problem for each agent $i\in \mathcal V $. Then, given the initial state $x_i(t)$, the DMPC problem for the $i$th agent at timestamp $t$ can be written as
\begin{IEEEeqnarray*} {cl}
\label{eq:problem_init}
\min_{u_i,x_i} \quad  & \sum_{i\in \mathcal V} J_i(x_i,u_i) \IEEEyesnumber \IEEEyessubnumber \label{subeq:cost} \\ 
\operatorname{s.t.} \quad  & x_i(\tau+1) = A_ix_i(\tau) + B_i u_i(\tau) \IEEEnonumber \IEEEyessubnumber \label{subeq:dynamics}\\
& \underline u_{i} \leq u_i(\tau) \leq \overline u_{i} \IEEEnonumber \IEEEyessubnumber \label{subeq:control_bound}\\ 
&\underline x_{i} \leq x_i(\tau+1) \leq \overline x_{i} \IEEEnonumber \IEEEyessubnumber \label{subeq:state_bound}\\
& (x_i(\tau+1), x_j(\tau+1)) \in \mathbb C \IEEEnonumber \IEEEyessubnumber \label{subeq:couple_constr} \\
& \forall j\in \nu(i), \forall \tau\in \mathbb Z_{t}^{t+T-1},
\IEEEnonumber \IEEEnosubnumber
\end{IEEEeqnarray*}
\noindent{where} the subscript $(\cdot)_i$ means the corresponding variable of the $i$th agent, $x_i\in \mathbb R^{n}$ and $u_i\in \mathbb R^{m}$ denote the state and control input vectors of the $i$th agent, respectively. $A_i\in \mathbb R^{n\times n}, B_i\in \mathbb R^{n\times m}$ are the state matrix and control input matrix, respectively. $t$ denotes the initial time stamp and $T$ is the prediction horizon. Notice that \eqref{subeq:cost} is the objective function for the state and control variables; \eqref{subeq:dynamics} denotes the kinematics of the agents; \eqref{subeq:state_bound} and~\eqref{subeq:control_bound} denote the bound constraints of state and control input vectors, $\overline x_{i}$ ($\underline x_{i}$) and $\overline u_{i}$ ($\underline u_{i}$) are the upper (lower) bound of state and control input vectors, respectively; \eqref{subeq:couple_constr} denotes the coupling constraints of two interconnected agents (the $i$th agent and the $j$th agent with $j\in \nu(i)$) and $\mathbb C$ denotes the set of the coupling constraints. The cost function of the $i$th agent $J_i(x_i,u_i)$ is a summation of the stage cost term $\ell_i$ and the terminal cost term $\ell_{T,i}$, i.e., 
\begin{IEEEeqnarray*}{l}
\quad J_i(x_i,u_i)  \\
= \ell_i\left( x_i(t+1), \cdots, x_{i}(t+T-1), u_i(t), \cdots, u_i(t+T-1) \right) \\
+ \ell_{T,i}(x_i(t+T)), \IEEEeqnarraynumspace \yesnumber
\end{IEEEeqnarray*}
where the initial state $x_i(t)$ is known. The stage cost function $\ell_i$ is defined as
\begin{IEEEeqnarray*}{rCl}
	\ell_i(x_i, u_i) 
	&=& \sum_{\tau=t+1}^{t+T-1} \langle x_i(\tau) - x_{i,\text{ref}}(\tau),\: Q_i \left(x_i(\tau) - x_{i,\text{ref}}(\tau) \right) \rangle 
	\\
	&& + \sum_{\tau=t}^{t+T-1} \langle u_i(\tau) \:,\: R_iu_i(\tau) \rangle, \yesnumber
\end{IEEEeqnarray*}
where $x_{i,\mathrm{ref}}$ is the reference state the $i$th agent tries to follow, and $Q_i\in \mathbb R^{n\times n}$ and $R_i\in \mathbb R^{m\times m}$ are the diagonal weighting matrices for the state and control input vectors, respectively. 
The terminal cost function $\ell_{T,i}$ is 
\begin{IEEEeqnarray*}{rCl}
\ell_{T,i}(x_i(t+T)) 
	&=& \langle x_i(t+T) - x_{i,\text{ref}}(t+T) \:,\: \\ 
	&& \:\: Q_{i,T} \left(x_i(t+T) - x_{i,\text{ref}}(t+T) \right) \rangle,	  \yesnumber
\end{IEEEeqnarray*}
where $Q_{i,T}$ is the diagonal weighting matrix for the terminal state vector. 

In a specific multi-agent system, $\mathbb C$ can be specified as 
\begin{IEEEeqnarray*}{c}
\label{eq:constr_coupling}
\mathbb C = \Big\{ (x_i(\tau),x_j(\tau)) \: \Big | \: d_{\mathrm{safe}} \le \left\|L_p(x_i(\tau)-x_j(\tau)) \right\| \le d_{\mathrm {cmu}} , \\
 \quad \forall i\in \mathcal V,  \forall j\in \nu(i), \forall \tau\in \mathbb Z_{t+1}^{t+T} \Big\},	 \yesnumber
\end{IEEEeqnarray*}
where $L_p$ is the locating matrix to extract the position vector $p_i$ from the state vector $x_i$, i.e., $p_i(\tau) = L_p x_i(\tau)$.

\begin{assumption}
	$\nu(i)$ in the prediction horizon $[t+1,t+T]$ is time-invariant.
\end{assumption}



\section{Problem Transformation}
\label{section:problem_reformulation}
\subsection{Problem Reformulation}
First of all, the decision variable is set as 
\begin{IEEEeqnarray*}{rCl}
z_i &=& \big(u_i(t),x_i(t+1),u_i(t+1),x_i(t+2), \cdots, \\ 
&& \quad u_i(t+T-1), x_i(t+T) \big) \in \mathbb R^{T(m+n)}.\yesnumber
\end{IEEEeqnarray*}
It is straightforward that $L_px_i=p_i \in \mathbb R^{n_p}, \forall i\in \mathcal{V}$, where $n_p$ is the dimension of position in state vector $p_i$ and $n_p\leq n$, and $\|p_i(\tau)- p_j(t)\|^2 = \left\langle p_i(\tau)- p_j(t) \:,\: p_i(\tau)- p_j(t) \right\rangle$. For all $j\in \nu(i)$ and $\tau \in \mathbb Z_{t+1}^{t+T}$, the coupling constraints $d_{\mathrm{safe}} \le \|p_i(\tau)-p_j(\tau)\| \le d_{\mathrm {cmu}}$, $\forall j\in \nu(i)$, $\forall i\in \mathcal V$, $\forall \tau\in \mathbb Z_{t+1}^{t+T}$ can be equivalently expressed as
\begin{IEEEeqnarray*}{c}
		d_{\mathrm{safe}}^2 \leq  \left( p_i(t+1)- p_j(t)\right)^\top \left(p_i(t+1)- p_j(t) \right)  \leq d_{\mathrm{cmu}}^2, \\
		\quad \quad \quad \quad \quad \quad \quad \quad \quad \quad \quad \quad \quad \quad \quad \quad \forall j \in \nu(i),
	\\ \vdots \\
		d_{\mathrm{safe}}^2 \leq  \left( p_i(t+T)- p_j(t)\right)^\top \left(p_i(t+T)- p_j(t) \right)  \leq d_{\mathrm{cmu}}^2, \\
		\quad \quad \quad \quad \quad \quad \quad \quad \quad \quad \quad \quad \quad \quad \quad \quad \forall j \in \nu(i),
	\yesnumber
\end{IEEEeqnarray*}
where $r_i$ is the number of the neighbours of the $i$th agent, i.e., $r_i =|\nu(i)|$ where $|\cdot|$ is the cardinality operator for the set $\nu(i)$. This means the $i$th agent has $r_i$ neighbors, and there are $r_i$ inequalities for each $\tau$. Thus, the coupling constraint $d_{\mathrm{safe}} \le \|p_i(\tau)-p_j(\tau)\| \le d_{\mathrm {cmu}}$, $\forall j\in \nu(i)$, $\forall i\in \mathcal V$, $\forall \tau\in \mathbb Z_{t+1}^{t+T}$ denotes $Tr_i$ inequalities for all $\tau$ and $j$.

Define a known vector $[p_j] = \left\{p_j(t)\right\}_{\forall j\in \nu(i)} \in \mathbb R^{r_in_p}$. Notice that the position vector of all neighbors of the $i$th agent $p_j(t)$ are known due to the existence of communication. Thus, the coupling constraints can be rewritten as 
\begin{IEEEeqnarray*}{rCl}
	\mathcal{M}\left(z_i; [p_j] \right) \mathcal{P}\left(z_i; [p_j] \right) &\leq& d_{\mathrm{cmu}}^2\\ 
	\mathcal{M}\left(z_i; [p_j] \right) \mathcal{P}\left(z_i; [p_j] \right) &\geq& d_{\mathrm{safe}}^2, \IEEEeqnarraynumspace \yesnumber
\end{IEEEeqnarray*}
with the mapping $\mathcal P: \mathbb R^{T(m+n)}\rightarrow \mathbb R^{Tr_in_p}$ characterized by
\begin{IEEEeqnarray*}{rCl} 
	\mathcal P\left(z_i; [p_j] \right) = (I_T\otimes1_{r_i}\otimes I_{n_p})M_pz_i - 1_{T} \otimes [p_j], \IEEEeqnarraynumspace \yesnumber
\end{IEEEeqnarray*}
where the matrix $M_p \in \mathbb R^{Tn_p\times T(m+n)}$ is a matrix to extract all position vectors $\left(p_i(t+1),p_i(t+2),\cdots,p_i(t+T)\right)$ from the decision variable $z_i$, and $M_p=I_T \otimes \begin{bmatrix} {0}_{(n_p,m)} & I_{n_p} & {0}_{(n_p,(n-n_p))} \end{bmatrix}$. 
Besides, the mapping $\mathcal M: \mathbb R^{T(m+n)}\rightarrow \mathbb R^{Tr_i \times Tr_in_p} $ is given by
\begin{IEEEeqnarray*}{rCl}
	\mathcal M\left(z_i; [p_j] \right)
	= \left( I_{Tr_i}\otimes 1_{n_p}^\top \right) \left( I_{Tr_in_p}  \odot \left(  {\mathcal P\left(z_i; [p_j] \right)} {1}_{Tr_in_p}^\top \right) \right).  \\ \yesnumber
\end{IEEEeqnarray*}

Therefore, the large-scale cooperative optimization problem~\eqref{eq:problem_init} can be reformulated as
\begin{IEEEeqnarray*}{rCl}
\label{eq:problem_simp1}
\min\limits_{z_i} &\quad& \sum_{i\in \mathcal V} J_i(z_i)  \IEEEnonumber  \\ 
\operatorname{s.t.} &&  G_iz_i = g_i \IEEEnonumber  \label{subeq:dynamics_compact}\\
&& F_{z,i} z_i \leq f_{z,i} \IEEEnonumber \IEEEnosubnumber \label{subeq:bound_compact}\\
&&  {\mathcal{M}\left(z_i; [p_j] \right) \mathcal{P}\left(z_i; [p_j] \right)} \leq d_{\mathrm{cmu}}^2 \IEEEnonumber \\ 
&& {\mathcal{M}\left(z_i; [p_j] \right) \mathcal{P}\left(z_i; [p_j] \right)} \geq d_{\mathrm{safe}}^2 \IEEEnonumber \\ 
&& \forall j\in \nu(i), \forall \tau\in \mathbb Z_{t+1}^{t+T},  \IEEEyesnumber 
\end{IEEEeqnarray*}
where $J_i(z_i) = \langle z_i-z_{i,\text{ref}} \:,\: H_i (z_i-z_{i,\text{ref}}) \rangle$; $H_i =$ $\operatorname{blockdiag}(R_i, Q_i, R_i, Q_i, \cdots,R_i, Q_i, R_i, Q_{T,i})$ is the weighting matrix for the optimization variable $z_i$; $z_{i,\text{ref}}$ is the reference vector, i.e., $z_{i,\text{ref}} =\big(0_{m}, x_{i,\text{ref}}(t+1),$  $0_{m}, x_{i,\text{ref}}(t+2), \cdots, 0_{m}, x_{i,\text{ref}}(t+T)\big)$. Due to the fact that the agents' dynamic models are linear time-invariant during the optimization horizon $[t,t+T]$, we can rewrite~\eqref{subeq:dynamics} as the equality constraint $G_iz_i = g_i$ in~\eqref{eq:problem_simp1},
where
\begin{IEEEeqnarray*}{rCl}
	G_i &=& I_T \otimes \begin{bmatrix}-B_i & I_n \end{bmatrix}  \; + \operatorname{blockdiag} \Big(0_{(n,m)}, \\
	&&\underbrace{\begin{bmatrix} -A_i & 0_{(n,m)} \end{bmatrix}, \cdots, \begin{bmatrix} -A_i & 0_{(n,m)} \end{bmatrix}}_{T-2},  \begin{bmatrix} -A_i & 0_{(n,m+n)} \end{bmatrix} \Big) \IEEEeqnarraynumspace
\end{IEEEeqnarray*}
and $g_i = \Big[ \left(A_ix_i(t)\right)^\top \quad \underbrace{0_n^\top \quad 0_n^\top \quad \cdots \quad 0_n^\top}_{T-1} \Big]^\top$ with a known state vector $x_i(t)$ at current timestamp $t$. 
\noindent Notice that $G_i\in \mathbb R^{Tn\times T(m+n)}$ and $g_i\in\mathbb R^{Tn}$. Additionally, \eqref{subeq:state_bound} and \eqref{subeq:control_bound} can be rewritten as $F_{z,i} z_i \leq f_{z,i}$ in~\eqref{eq:problem_simp1}, where
\begin{IEEEeqnarray*}{rCl}
F_{z,i} = \begin{bmatrix} I_{T(m+n)} \\ -I_{T(m+n)} \end{bmatrix},
f_{z,i} = \begin{bmatrix} { 1_{T} \otimes \begin{bmatrix} \overline u_{i}^\top & \overline x_i^\top \end{bmatrix}^\top } \\ {1_{T} \otimes \begin{bmatrix} -\underline u_{i}^\top & -\underline x_i^\top \end{bmatrix}^\top } \end{bmatrix}.
\end{IEEEeqnarray*}

\begin{remark}
	It is well-known that increasing the prediction horizon $T$ can enlarge the domain of attraction of the MPC controller. However, this will increase the computational burden. Weighting the terminal cost can also enlarge the domain of attraction of the MPC controller without the occurrence of the terminal constraints; thus, the stabilizing weighting factor of a given initial state can be included, which has been proved in~\cite{limon2006stability} that the asymptotic stability of the nonlinear MPC controller can be achieved.
\end{remark}

By introducing an auxiliary variable $z_{f,i} \in \mathbb R^{2T(m+n)}$, we have
\begin{IEEEeqnarray*}{rCl}
\label{eq:problem_simp2}
\min\limits_{z_i} &\quad& \sum_{i\in \mathcal V} J_i(z_i)  + \delta_{\mathbb R_+^{2T(m+n)}}(z_{f,i}) \\ 
\operatorname{s.t.} &&  \mathcal A_iz_i + \mathcal B_iz_{f,i} - h_i = 0 \\
&& {\mathcal{M}\left(z_i; [p_j] \right) \mathcal{P}\left(z_i; [p_j] \right)} \leq d_{\mathrm{cmu}}^2  \\ 
&& {\mathcal{M}\left(z_i; [p_j] \right) \mathcal{P}\left(z_i; [p_j] \right)} \geq d_{\mathrm{safe}}^2 \\
&& \forall j\in \nu(i),  \yesnumber 
\end{IEEEeqnarray*}
where $\mathcal A_i = \begin{bmatrix} G_i \\ F_{z,i}  \end{bmatrix}$, $\mathcal B_i = \begin{bmatrix} 0_{(Tn,2T(m+n))} \\ I_{2T(m+n)}  \end{bmatrix}$, and $h_i = \begin{bmatrix} g_i \\ f_{z,i} \end{bmatrix}$, $\delta_{\mathbb R_+^{2T(m+n)}}(z_{f,i})$ is an indicator function defined as
\begin{IEEEeqnarray*}{rCl}
	\delta_{\mathbb R_+^{2T(m+n)}}(z_{f,i}) = 
	\begin{cases}
		0, & \text{if } z_{f,i} \in \mathbb R_+^{2T(m+n)} \\
		\infty, & \text{otherwise}.
	\end{cases}\IEEEyesnumber
\end{IEEEeqnarray*}
\noindent Since the second inequality in~\eqref{eq:problem_simp2} is nonlinear and nonconvex, we define a set 
\begin{IEEEeqnarray*}{rl}
\mathcal Z_i = \left\{  z_i\in \mathbb R^{T(m+n)} \middle | 
\begin{array}{l}
\mathcal M(z_i) \mathcal P(z_i) \leq d_{\mathrm{cmu}}^2  \\
\mathcal M(z_i)\mathcal P(z_i) \geq d_{\mathrm{safe}}^2
\end{array} 
\right\}. \IEEEeqnarraynumspace \IEEEyesnumber
\end{IEEEeqnarray*}
Then, the problem~\eqref{eq:problem_simp2} can be simplified as 
\begin{IEEEeqnarray*}{rCl}
\label{eq:problem_simp4}
\min\limits_{z_i} &\quad&  J_i(z_i) + \delta_{\mathcal Z_i}(z_i) + \delta_{\mathbb R_+^{2T(m+n)}}(z_{f,i})  \\ 
\operatorname{s.t.} &&  \mathcal A_iz_i + \mathcal B_iz_{f,i} - h_i = 0 ,  \yesnumber
\end{IEEEeqnarray*}
where $\delta_{\mathcal Z_i}(z_i)$ is an indicator function of $z_i$ onto the set $\mathcal Z_i$. 

\subsection{Optimization Problem Transformation}
Due to the fact that the direct use of the conventional ADMM to the nonconvex optimization problem~\eqref{eq:problem_simp4} directly cannot ensure the convergence, problem transformation is significant in the constrained nonconvex optimization problem~\cite{jiang2019structured}. Thus, a slack variable $s_i\in \mathbb R^{2Tm+3Tn}$ is introduced to the problem~\eqref{eq:problem_simp4}, and then the transformed optimization problem is given by
\begin{IEEEeqnarray*}{rCl}
\label{eq:problem_relax}
\min\limits_{z_i} &\quad & J_i(z_i) + \delta_{\mathcal Z_i}(z_i) + \delta_{\mathbb R_+^{2T(m+n)}}(z_{f,i})  \\ 
\operatorname{s.t.} &&  \mathcal A_iz_i + \mathcal B_iz_{f,i} - h_i + s_i = 0  \\
&& s_i = 0. \yesnumber
\end{IEEEeqnarray*}
After introducing $s_i$, the linear coupling constraints in~\eqref{eq:problem_simp4} change from 2 blocks ($z_i$ and $z_{f,i}$) to 3 blocks ($z_i$, $z_{f,i}$ and $s_i$). Since the block about $s_i$ is an identity matrix $I_{2Tm+3Tn}$ whose image is the whole space, there always exists an $s_i$ such that $\mathcal A_iz_i + \mathcal B_iz_{f,i} - h_i + s_i = 0$ is satisfied, given any $z_i$ and $z_{f,i}$. In addition, the constraint $s_i=0$ can be treated separately from the equality constraint $\mathcal A_iz_i + \mathcal B_iz_{f,i} - h_i + s_i = 0$. Therefore, we can consider the constraints in the problem~\eqref{eq:problem_relax} by separating them into two groups of constraints. The first group considers the equality constraint and the nonconvex constraint involved by the indicator function $\delta_{\mathcal Z_i}(z_i)$ and $\delta_{\mathbb R_+^{2T(m+n)}}(z_{f,i})$, and the other group deals with the constraint $s_i=0$. As for the first sub-problem which does not include the constraint $s_i=0$, some existing techniques of the ADMM can be employed. On the other hand, the constraint $s_i=0$ can be transformed by dualizing the constraint $s_i=0$ with $\lambda_i$ and adding a quadratic penalty $\frac{\beta_i}{2} \|s_i\|^2$. Then the problem~\eqref{eq:problem_relax} can be rewritten as 
\begin{IEEEeqnarray*}{rCl}
\label{eq:problem_relax2}
\min\limits_{z_i, z_{f,i}, s_i} &\quad & J_i(z_i) + \delta_{\mathcal Z_i}(z_i) + \delta_{\mathbb R_+^{2T(m+n)}}(z_{f,i})  \\
&& + \langle \lambda_i, s_i \rangle + \frac{\beta_i}{2} \|s_i\|^2  \\ 
\operatorname{s.t.} &&  \mathcal A_iz_i + \mathcal B_iz_{f,i} - h_i + s_i = 0 , \yesnumber
\end{IEEEeqnarray*}
where $\lambda_i\in\mathbb R^{2Tm+3Tn}$ is the dual parameter and $\beta_i\in \mathbb R_+$ is the penalty parameter.
Then, the ALM is used to solve the second sub-problem. 

The augmented Lagrangian function of problem~\eqref{eq:problem_relax2} is 
\begin{IEEEeqnarray*}{rCl}
\mathcal L_{\rho_i} &=& J_i(z_i) + \delta_{\mathcal Z_i}(z_i) + \delta_{\mathbb R_+^{2T(m+n)}}(z_{f,i}) + \langle \lambda_i, s_i \rangle \\
&& + \frac{\beta_i}{2} \|s_i\|^2 + \langle y_i, \mathcal A_iz_i + \mathcal B_iz_{f,i} - h_i + s_i\rangle \\
&& + \frac{\rho_i}{2}\| \mathcal A_iz_i + \mathcal B_iz_{f,i} - h_i + s_i \|^2, \yesnumber
\end{IEEEeqnarray*}
\noindent where $y_i\in \mathbb R^{2Tm+3Tn}$ is the dual parameter of constraint $\mathcal A_iz_i + \mathcal B_iz_{f,i} - h_i + s_i=0$ and $\rho_i\in \mathbb R$ is the penalty parameter. Then, the first-order optimality conditions at a stationary solution $(z_i^k, z_{f,i}^k, s_i^k, y_i^k)$ are given by
\begin{IEEEeqnarray*}{lCl}
\label{eq:first_order_optimality_condition}
0 \in \nabla J_i(z_i^k) + \mathbb N_{\mathcal Z_i}(z_i^k) + \mathcal A_i^\top y_i^k \IEEEyesnumber \IEEEyessubnumber \label{subeq:optimality_condition_z_i} \\
0 \in \mathcal B_i^\top y_i^k + \mathbb N_{\mathbb R_+^{2T(m+n)}} \left(z_{f,i}^k\right) \nonumber \IEEEyessubnumber \label{subeq:optimality_condition_z_fi}\\
\lambda_i^k + \beta^k s_i^k + y_i^k = 0 \nonumber \IEEEyessubnumber \label{subeq:optimality_condition_s_i}\\
\mathcal A_iz_i^k + \mathcal B_iz_{f,i}^k - h_i + s_i^k = 0  \label{subeq:optimality_condition_y_i}, \nonumber \IEEEyessubnumber
\end{IEEEeqnarray*}
where $\mathbb N_{\mathcal Z_i}(z_i^k)$ and $\mathbb N_{\mathbb R_+^{2T(m+n)}} \left(z_{f,i}^k\right)$ denote the normal cone of the set $\mathcal Z_i$ and $\mathbb R_+^{2T(m+n)}$ with respect to $z_i^k$ and $z_{f,i}^k$, respectively, and the superscript $(\cdot)^k$ means the variable in the $k$th iteration. 
\begin{remark}
\label{remark1}
	A solution that satisfies the optimality conditions~\eqref{eq:first_order_optimality_condition} may not necessarily satisfy the primal feasibility of the constraint $s_i=0$. However, ALM can prompt the slack variable $s_i$ to zero through updating its dual parameter $\lambda_i$.
\end{remark}
Based on Remark~\ref{remark1}, the update of $\lambda_i$ helps to prompt $s_i$ to zero with
$\lambda_i^{k+1} = \lambda_i^k + \beta_i^k s_i^k$.

\section{Hierarchical ADMM}
\label{section:two_level_ADMM}
For the problem~\eqref{eq:problem_relax2}, we treat it into two groups. The first one is to solve the problem~\eqref{eq:problem_relax2} and update $z_i, z_{f,i}, s_i, y_i$ in sequence with keeping $\lambda_i$ and $\beta_i$ constant, and the other group is to update the parameter $\lambda_i$ and $\beta_i$. In order to clarify the two groups, the iteration number of the first group (inner iteration) is indexed by $r$, and the one of the second group (outer iteration) is $k$.

\noindent\textbf{1. Update $z_i$:}
\begin{IEEEeqnarray*}{rCl}
\label{eq:update_z_i}
z_i^{k,r+1} &=& \underset{z_i}{\operatorname{argmin}} \; \mathcal L_{\rho_i^k} \left(z_i, z_{f,i}^{k,r}, s_i^{k,r}; y_i^{k,r}, \lambda_i^k, \rho_i^k, \beta_i^k \right) \\
&=& \underset{z_i\in \mathcal Z_i}{\operatorname{argmin}} \;  J_i(z_i) + \frac{\rho_i^k}{2} \left\| \mathcal A_iz_i + \mathcal B_iz_{f,i}^{k,r} \right. \\
&& \left. \quad \quad + s_i^{k,r} - h_i+\frac{y_i^{k,r}}{\rho_i^{k}} \right\|^2. \yesnumber 
\end{IEEEeqnarray*}

\noindent\textbf{2. Update $z_{f,i}$:}
\begin{IEEEeqnarray*}{rCl}
\label{eq:update_z_f_i_ori}
z_{f,i}^{k,r+1} &=& \underset{z_{f,i}}{\operatorname{argmin}} \;\mathcal L_{\rho_i^k} \left(z_i^{k,r+1}, z_{f,i}, s_i^{k,r}; y_i^{k,r}, \lambda_i^k, \rho_i^k, \beta_i^k \right) \\
&=& \underset{z_{f,i}}{\operatorname{argmin}} \; \delta_{\mathbb R_+^{2T(m+n)}}(z_{f,i}) + \frac{\rho_i^k}{2} \left\| \mathcal A_iz_i^{k,r+1} \right.  \\
&& \left. \quad + \mathcal B_iz_{f,i} + s_i^{k,r} - h_i+\frac{y_i^{k,r}}{\rho_i^{k}} \right\|^2. \yesnumber
\end{IEEEeqnarray*}
In order to obtain a close-form solution to the problem~\eqref{eq:update_z_f_i_ori}, we define $\mathcal T_i^k\in \mathbb S^{2T(m+n)}$ as a given positive semidefinite linear operator. Then, it follows that
\begin{IEEEeqnarray*}{rCl}
z_{f,i}^{k,r+1} &=& \arg\min\limits_{z_{f,i}} \delta_{\mathbb R_+^{2T(m+n)}}(z_{f,i}) + \frac{1}{2} \left\|z_{f,i} - z_{f,i}^{k,r}\right\|^2_{\mathcal T_i^k}  \\
&& + \frac{\rho_i^k}{2} \left\| \mathcal A_iz_i^{k,r+1} + \mathcal B_iz_{f,i} + s_i^{k,r} - h_i+\frac{y_i^{k,r}}{\rho_i^{k}} \right\|^2. \IEEEeqnarraynumspace \yesnumber
\end{IEEEeqnarray*}
Then, the optimality condition to the sub-problem of $z_{f,i}$ updating is given by
\begin{IEEEeqnarray*}{rCl}
0 &\in & \partial\delta_{\mathbb R_+^{2T(m+n)}}(z_{f,i}) + \rho^k\mathcal B_i^T\mathcal B_iz_{f,i} + \mathcal T_{i}^{k}\left(z_{f,i} - z_{f,i}^{k,r}\right) \\
&&+ \rho^k\mathcal B_i^T\left( \mathcal A_iz_i^{k,r+1} + s_i^{k,r} - h_i+\frac{y_i^{k,r}}{\rho_i^{k}} \right).  \yesnumber
\end{IEEEeqnarray*}
Setting $\mathcal T_{i}^{k}=\alpha_i^{k} I_{2T(m+n)} - \rho_i^k\mathcal B_i^T\mathcal B_i$ where $\alpha_i^{k}$ is the largest eigenvalue of the matrix $\rho_i^k\mathcal B_i^T\mathcal B_i$, we have
\begin{IEEEeqnarray*}{rCl}
0&\in& \left(\partial\delta_{\mathbb R_+^{2T(m+n)}} + \alpha_i^{k} I_{2T(m+n)} \right)z_{f,i}  - \left(\alpha_i^{k}I - \rho_i^k\mathcal B_i^T\mathcal B_i\right)z_{f,i}^{k,r} \\
&& + \rho^k\mathcal B_i^T\left( \mathcal A_iz_i^{k,r+1} + s_i^{k,r} - h_i+\frac{y_i^{k,r}}{\rho_i^{k}} \right).   \yesnumber
\end{IEEEeqnarray*}
Then, it follows that
\begin{IEEEeqnarray*}{lll}
z_{f,i}^{k,r+1} 
&=& \left({\alpha_i^{k}}^{-1}\partial\delta_{\mathbb R_+^{2T(m+n)}} + I_{2T(m+n)} \right)^{-1} \\
&&  \Bigg( \left(I_{2T(m+n)} - {\alpha_i^{k}}^{-1}\rho_i^k\mathcal B_i^T\mathcal B_i\right)z_{f,i}^{k,r} \\
&&  \; - {\alpha_i^{k}}^{-1}\rho^k\mathcal B_i^T\left( \mathcal A_iz_i^{k,r+1} + s_i^{k,r} - h_i+\frac{y_i^{k,r}}{\rho_i^{k}} \right) \Bigg). \\ \yesnumber
\end{IEEEeqnarray*}

\begin{lemma}
\label{lemma:projection}
The projection operator $\operatorname{Proj}_{\mathcal C}(\cdot)$ with respect to the convex set $\mathcal C$ can be expressed as 
\begin{IEEEeqnarray}{rCl}
	\operatorname{Proj}_{\mathcal C}(\cdot) = (I+\alpha\partial\delta_{\mathcal C})^{-1},
\end{IEEEeqnarray}
where $\partial$ is the sub-differential operator, and $\alpha \in \mathbb R$ is any arbitrary number.	
\end{lemma}
\begin{proof}
	Define a finite dimensional Euclidean space $\mathcal X$ equipped with an inner product and its induced norm such that $\mathcal C \subset \mathcal X$. For any $x\in \mathcal X$, there exists $z\in\mathcal X$ such that $z\in (I+\alpha\partial\delta_{\mathcal C})^{-1}(x)$. Then, it follows that 
	\begin{IEEEeqnarray*}{rCl}
		x &\in& (I+\alpha\partial\delta_{\mathcal C})(z) 
		= z+\alpha\partial\delta_{\mathcal C}. \yesnumber
	\end{IEEEeqnarray*}
	The projection operator $\operatorname{Proj}_{\mathcal C}(z)$ is 
	\begin{IEEEeqnarray}{rCl}
	\label{eq:projection_lemma}
		\operatorname{Proj}_{\mathcal C}(z) &=& \underset{z}{\operatorname{argmin}}\left\{ \delta_{\mathcal C}(x) +\frac{1}{2\alpha} \|z-x\|^2 \right\}.
	\end{IEEEeqnarray}
	Due to the fact that the optimization problem~\eqref{eq:projection_lemma} is strictly convex, the sufficient and necessary condition of~\eqref{eq:projection_lemma} is 
	\begin{IEEEeqnarray}{rCl}
	\label{eq:projection_lemma2}
	0 \in \alpha\partial\delta_{\mathcal C}(x) + z - x.
	\end{IEEEeqnarray}
	Since~\eqref{eq:projection_lemma2} is equivalent to~\eqref{eq:projection_lemma} and the projection onto a convex set is unique, the operator $(I+\alpha\partial\delta_{\mathcal C})^{-1}$ is a point-to-point mapping. 
	This completes the proof of Lemma~\ref{lemma:projection}.
\end{proof}
Based on Lemma~\ref{lemma:projection}, $z_{f,i}$ can be determined as
\begin{IEEEeqnarray*}{rCl}
\label{eq:update_z_f_i}
z_{f,i}^{k,r+1}
&= \operatorname{Proj}_{\mathbb R_+^{2T(m+n)}} \Bigg\{  \left(I_{2T(m+n)} - {\alpha_i^{k}}^{-1}\rho_i^k\mathcal B_i^T\mathcal B_i\right)z_{f,i}^{k,r} \\
& - {\alpha_i^{k}}^{-1}\rho^k\mathcal B_i^T\left( \mathcal A_iz_i^{k,r+1} + s_i^{k,r} - h_i+\frac{y_i^{k,r}}{\rho_i^{k}} \right) \Bigg\}, \IEEEeqnarraynumspace \yesnumber
\end{IEEEeqnarray*}
where the projection operator $\operatorname{Proj}_{\mathbb R_+^{2T(m+n)}} (z)$ means the projection of $z$ onto a closed convex set $\mathbb R_+^{2T(m+n)}$.

\begin{lemma}
\label{lemma:project_reals}
With $\mathcal C=\mathbb R_+^n$, then for any $z\in \mathbb R^n$, the projection of $z$ onto $\mathcal C$ is given by
\begin{IEEEeqnarray}{rCl}
\operatorname{Proj}_{\mathcal C}(z) = \max\{z,0\}.
\end{IEEEeqnarray}
\end{lemma}

According to Lemma~\ref{lemma:project_reals}, $z_{f,i}^{k,r+1}$ can be easily obtained.

\noindent\textbf{3. Update $s_{i}$:}
\begin{IEEEeqnarray*}{rCl}
\label{eq:update_s_i}
s_{i}^{k,r+1} &=& \underset{s_{i}}{\operatorname{argmin}} \; \mathcal L_{\rho_i^k} \left(z_i^{k,r+1}, z_{f,i}^{k,r+1}, s_i; y_i^{k,r}, \lambda_i^k, \rho_i^k, \beta_i^k \right) \\
&=& -\frac{\rho_i^k}{\rho_i^k+\beta_i^k}\left( \mathcal A_iz_i^{k,r+1} + \mathcal B_iz_{f,i}^{k,r+1}- h_i+\frac{y_i^{k,r}}{\rho_i^{k}} \right) \\
 && \; - \frac{\lambda_i^k}{\rho_i^k+\beta_i^k}. \yesnumber
\end{IEEEeqnarray*}

\noindent\textbf{4. Update $y_{i}$:}
\begin{IEEEeqnarray*}{rCl}
\label{eq:update_y_i}
y_i^{k,r+1} &=& \underset{y_{i}}{\operatorname{argmin}} \; \mathcal L_{\rho_i^k} \left(z_i^{k,r+1}, z_{f,i}^{k,r+1}, s_i^{k,r+1}; y_i, \lambda_i^k, \rho_i^k, \beta_i^k \right) \\
&=& y_i^{k,r} + \rho_i^k \left( \mathcal A_iz_i^{k,r+1} + \mathcal B_iz_{f,i}^{k,r+1} + s_i^{k,r+1} - h_i \right). \\  \yesnumber
\end{IEEEeqnarray*}

\begin{assumption}
	\label{assumption:z-update_descent}
	Given $\lambda_i^k$, $\beta_i^k$ and $\rho_i^k$, the $z_i$-update can find a stationary solution $z_i^r$ such that 
	\begin{IEEEeqnarray*}{rCl}
	0 &\in & \partial_{z_i} \mathcal L_{\rho_i^k}(z_i^r, z_{f,i}^{r-1}, s_i^{r-1},y_i^{r-1}) \\
	\mathcal L_{\rho_i^k}(z_i^r, z_{f,i}^{r-1}, s_i^{r-1},y_i^{r-1}) &\leq & \mathcal L_{\rho_i^k}(z_i^{r-1}, z_{f,i}^{r-1}, s_i^{r-1},y_i^{r-1}), \\ \yesnumber
	\end{IEEEeqnarray*}
\end{assumption}
\noindent for all $a\in \mathbb Z_{+}$.


\begin{lemma}
	\label{lemma2}
	For all $a\in \mathbb Z_{+}$, the following conditions holds:
	\begin{IEEEeqnarray*}{c}
		\left \langle \mathcal{B}_i^\top y_i^{r-1} + \rho_i\mathcal{B}_i^\top \left(\mathcal A_iz_i^{r} + \mathcal B_iz_{f,i}^{r} + s_i^{r-1} - h_i\right), \right. \\
		\left. (z_{f,i}-z_{f,i}^{r}) \right \rangle \geq 0, \; \forall z_{f,i} \in \mathbb R_{+}^{2T(m+n)} \\
		\lambda_i + \beta_is_i^r + y_i^r = 0.  \yesnumber
	\end{IEEEeqnarray*}
\end{lemma}
\begin{proof}
	This lemma can be straightforwardly derived based on the optimality conditions of $z_i$- and $z_{f,i}$-updates, i.e., \eqref{subeq:optimality_condition_z_fi} and~\eqref{subeq:optimality_condition_s_i}, and thus the proof is omitted.
\end{proof}

\begin{theorem}~\label{thm1}
	With $\rho_i^k \geq \sqrt{2} \beta_i^k$, the inner iteration of problem~\eqref{eq:problem_relax2} converges to the stationary point $(z_i^k, z_{f,i}^k, s_i^k, y_i^k)$ of the transformed problem~\eqref{eq:problem_relax2}.
\end{theorem}
\begin{proof}
	With Assumption~\ref{assumption:z-update_descent}, we have 
\begin{IEEEeqnarray*}{rCl}
\mathcal{L}_{\rho_i^k} \left(z_i^{r-1}, z_{f,i}^{r-1}, s_i^{r-1}, y_i^{r-1} \right) \geq \mathcal{L}_{\rho_i^k} \left(z_i^{r}, z_{f,i}^{r-1}, s_i^{r-1}, y_i^{r-1} \right). \\ \yesnumber
\end{IEEEeqnarray*}
Besides, we also have 
\begin{IEEEeqnarray*}{rCl}
\label{eq:theorem_subeq}
	&&\mathcal{L}_{\rho_i^k} \left(z_i^{r}, z_{f,i}^{r-1}, s_i^{r-1}, y_i^{r-1} \right) - \mathcal{L}_{\rho_i^k} \left(z_i^{r}, z_{f,i}^{r}, s_i^{r-1}, y_i^{r-1} \right) \\
	&=& \Big\langle \mathcal{B}_i^\top y_i^{r-1} + \rho_i^k\mathcal{B}_i^\top \left(\mathcal A_iz_i^{r} + \mathcal B_iz_{f,i}^{r} + s_i^{r-1} - h_i\right),  \\
	&&   \quad z_{f,i}^{r-1}-z_{f,i}^{r} \Big\rangle + \frac{\rho_i^k}{2} \left\|\mathcal{B}_i(z_{f,i}^{r-1}-z_{f,i}^{r})\right\|^2 \\
	&\geq & \frac{\rho_i^k}{2} \left\|\mathcal{B}_i(z_{f,i}^{r-1}-z_{f,i}^{r})\right\|^2. \yesnumber
\end{IEEEeqnarray*}
Note that the equality in~\eqref{eq:theorem_subeq} holds because of the following claim:
\begin{IEEEeqnarray*}{l}
 \left\| \mathcal A_iz_i^{r} + s_i^{r-1} - h_i + \mathcal B_iz_{f,i}^{r-1}\right\|^2 - \left\| \mathcal A_iz_i^{r} + s_i^{r-1} - h_i + \mathcal B_iz_{f,i}^{r} \right\|^2 \\
 = 2 \left\langle \mathcal A_iz_i^{r} + s_i^{r-1} - h_i + \mathcal B_iz_{f,i}^{r}, \mathcal B_iz_{f,i}^{r-1}-\mathcal B_iz_{f,i}^{r} \right\rangle \\
  \quad + \left\| \mathcal B_iz_{f,i}^{r-1}-\mathcal B_iz_{f,i}^{r} \right\|^2. \yesnumber
\end{IEEEeqnarray*}
In addition, the inequality in~\eqref{eq:theorem_subeq} can be obtained from Lemma~\ref{lemma2}. Now, the descent of $z_i$ and $z_{f,i}$ updates has been proved. Next, we will show the descent of $s_i$ and $y_i$ updates.
\begin{IEEEeqnarray*}{rCl}
\label{eq:theorem_subeq2}
	&&\mathcal{L}_{\rho_i^k} \left(z_i^{r}, z_{f,i}^{r}, s_i^{r-1}, y_i^{r-1} \right) - \mathcal{L}_{\rho_i^k} \left(z_i^{r}, z_{f,i}^{r}, s_i^{r}, y_i^{r} \right) \\
	&=& \left\langle \lambda_i^{k}, (s_{i}^{r-1}-s_{i}^{r}) \right\rangle + \frac{\beta_i^k}{2} \left(\left\|s_i^{r-1}\|^2-\|s_i^{r}\right\|^2\right) \\
	&&  \: + \left\langle y_i^{r}, s_{i}^{r-1}-s_{i}^{r} \right\rangle + \frac{\rho_i^k}{2} \left\| s_{i}^{r-1}-s_{i}^{r}\right\|^2 \\
	&&  \: - \rho_i^k \left\|\mathcal A_iz_i^{r} + \mathcal B_iz_{f,i}^{r} + s_i^{r} - h_i\right\|^2 \\
	&\geq & \left(\frac{\rho_i^k}{2}-\frac{{\beta_i^k}^2}{\rho_i^k}\right) \left\|s_{i}^{r-1}-s_{i}^{r}\right\|^2. \yesnumber
\end{IEEEeqnarray*}
Remarkably, the first equality holds due to~\eqref{eq:update_y_i} and the following derivation:
\begin{IEEEeqnarray*}{lCl}
	\label{eq:prove1}
	&&  -\rho_i^k \langle \mathcal A_iz_i^{r} + \mathcal B_iz_{f,i}^{r} - h_i + s_i^{r}, \mathcal A_iz_i^{r} + \mathcal B_iz_{f,i}^{r} - h_i + s_i^{r-1} \rangle \\
	&&  + \frac{\rho_i^k}{2} \Big( \left\|\mathcal A_iz_i^{r} + \mathcal B_iz_{f,i}^{r} - h_i + s_i^{r-1}\right\|^2 \\
	&&  \quad - \left\|\mathcal A_iz_i^{r} + \mathcal B_iz_{f,i}^{r} - h_i + s_i^{r}\right\|^2 \Big) \\
	&=& \frac{\rho_i^k}{2} \left\| s_i^{r-1} - s_i^{r} \right\|^2 - \rho_i^k \left\|\mathcal A_iz_i^{r} + \mathcal B_iz_{f,i}^{r} - h_i + s_i^{r} \right\|^2. \yesnumber
\end{IEEEeqnarray*}
To prove the inequality in~\eqref{eq:theorem_subeq2}, we define a function $f(s_i) = \langle \lambda_i^{r-1}, s_i \rangle + \frac{\beta_i}{2} \|s_i\|^2$. The gradient of $f(s_i)$ is $\nabla f(s_i^{r}) = \lambda_i^{r-1} + \beta_i s_i^{r} = -y_i^r$ due to~\eqref{subeq:optimality_condition_s_i}. It follows that $f(s_i^{r-1})-f(s_i^{r})+(y_i^r)^\top (s_i^{r-1} - s_i^r) \geq 0$, as $f(s_i)$ is convex. Therefore, the inequality in~\eqref{eq:theorem_subeq2} holds. 

Under the condition $\rho_i\geq \sqrt{2}\beta_i$, the descent of $\mathcal L_{\rho_i^k}$ with updating $(z_i^k, z_{f,i}^k, s_i^k, y_i^k)$ can be proved, which means the augmented Lagrangian function is sufficient descent. 

Since the function $f(s_i)$, a part of the augmented Lagrangian function, is Lipschitz differentiable with $\beta_i$, and $\lambda_i$ is bounded by $[\underline{\lambda_i},\overline{\lambda_i}]$, it is obvious that $\mathcal{L}_{\rho_i^k} \left(z_i^{r}, z_{f,i}^{r}, s_i^{r}, y_i^{r} \right)$ is lower bounded~\cite{sun2019two}. Thus, the solution of the inner loop of this transformed problem~\eqref{eq:problem_relax2} will converge to the stationary point $(z_i^k, z_{f,i}^k, s_i^k, y_i^k)$. This completes the proof of Theorem~\ref{thm1}.
\end{proof}

\noindent\textbf{5. Update $\lambda_i$ and $\beta_i$:}
Then, in the outer iterations, the dual variable $\lambda_i^k$ is updated. In order to drive the slack variable $s_i$ to zero, the penalty parameter $\beta_i^k$ is multiplied by a ratio $\gamma>1$ if the currently computed $s_i^k$ from the inner iterations does not decrease enough from the previous outer iteration $s_i^{k-1}$, i.e., $\|s_i^k\| > \omega \|s_i^{k-1}\|, \omega\in (0,1)$.
\begin{IEEEeqnarray*}{rCl}
\label{eq:update_lambda_beta}
\lambda_i^{k+1} &=& \operatorname{Proj}_{[\underline{\lambda}_i, \overline{\lambda}_i]} (\lambda_i^k + \beta_i^ks_i^k) \\
\beta^{k+1} &=& \begin{cases}
\gamma \beta_i^k, &\|s_i^k\| > \omega\|s_i^{k-1}\| \\
\beta_i^k, &\|s_i^k\| \leq \omega\|s_i^{k-1}\|,
\end{cases} \yesnumber
\end{IEEEeqnarray*}
where the projection operator $\operatorname{Proj}_{[\underline{\lambda}_i, \overline{\lambda}_i]} (\lambda_i^k + \beta_i^ks_i^k)$ means the projection of $\lambda_i^k + \beta_i^ks_i^k$ onto a closed convex set $[\underline{\lambda}_i, \overline{\lambda}_i]$.
\begin{lemma}
	\label{theorem:project_boundset}
	Given a hypercube set $[a,b]$ with $a,b\in \mathbb R^n$ and $a<b$, for any $z\in \mathbb R^n$, the projection of $z$ onto the hypercube $[a,b]$ is given by
	\begin{IEEEeqnarray}{rCl}~\label{eqn:mj}
		\operatorname{Proj}_{[a, b]} (z) &=& \begin{cases} z_{(i)}, &a_{(i)} \leq z_{(i)}\leq b_{(i)} \\
		    a_{(i)}, &z_{(i)} < a_{(i)} \\
		    b_{(i)}, &z_{(i)} > b_{(i)}, \\
		\end{cases}
	\end{IEEEeqnarray}
	where the subscript $\cdot_{(i)}$ means the $i$th component in vector $z$, $a$ and $b$. 
\end{lemma}

\begin{theorem}
\label{theorem:converge_outer}
Suppose that the point $(z_i^k, z_{f,i}^k, s_i^k, y_i^k)$ satisfies the optimality condition~\eqref{eq:first_order_optimality_condition}. For any $\epsilon_1, \epsilon_2, \epsilon_3 >0$, after finite outer iterations $k$, appropriate stationary point $(z_i^{k}, z_{f,i}^k, s_i^k, y_i^k)$ can be found such that the optimality condition of the original problem \eqref{eq:problem_simp4} before introducing the slack variable $s_i$ is also satisfied~\cite{sun2019two}, i.e., 
\begin{IEEEeqnarray*}{l}
\label{eq:eq_for_prove}
0 \in \nabla J_i(z_i^k) + \mathbb N_{\mathcal Z_i}(z_i^k) + \mathcal A_i^\top y_i^k \IEEEyesnumber \IEEEyessubnumber \label{subeq:primal_optimality_condition_z_i} \\
0 \in \mathcal B_i^\top y_i^k + \mathbb N_{\mathbb R_+^{2T(m+n)}} \left(z_{f,i}^k\right) \nonumber \IEEEyessubnumber \label{subeq:primal_optimality_condition_z_fi}\\
\mathcal A_iz_i^k + \mathcal B_iz_{f,i}^k - h_i = 0  \label{subeq:primal_optimality_condition_y_i}.  \IEEEyessubnumber
\end{IEEEeqnarray*}
\end{theorem}
\begin{proof}
	Assume that $(z_i^k, z_{f,i}^k, s_i^k, y_i^k)$ converges to the stationary point $(z_i^\star, z_{f,i}^\star, s_i^\star, y_i^\star)$. We have $z_i^\star \in \mathcal Z_i, z_{f,i}\in \mathbb R_+^{2T(m+n)}$ and $\mathcal A_i z_i^\star + \mathcal B_i z_{f,i}^\star - h_i + s_i^\star=0$.
	
	For the case where $\beta_i^k$ is bounded, we have $z_i^k\rightarrow 0$ and $z_i^\star=0$. For the case where $\beta_i^k$ is not bounded, based on the optimality condition~\eqref{eq:first_order_optimality_condition}, it follows that
	\begin{IEEEeqnarray}{rCl}
		\label{eq:subeq_prove_converge}
		\frac{{\lambda_i}^{k}}{\beta_i^{k}} + s_i^{k,r} + \frac{{y_i}^{k,r}}{\beta_i^{k}} = 0. 
	\end{IEEEeqnarray}
	Then, we have $s_i^\star = 0$ by taking the limitation of~\eqref{eq:subeq_prove_converge} to zero, as $y_i^k \rightarrow y_i^\star$ and $\lambda_i^k\in [\underline \lambda_i, \overline \lambda_i]$. 
	Since we have $s_i^\star = 0$ in both of the two cases, the optimality condition with respect to $y_i$, i.e., \eqref{subeq:primal_optimality_condition_y_i}, can be satisfied. Besides, the optimality condition with respect to $z_i$ and $z_{f,i}$, i.e., \eqref{subeq:primal_optimality_condition_z_i} and \eqref{subeq:primal_optimality_condition_z_fi}, can also be satisfied by taking $k\rightarrow 0$ on \eqref{subeq:optimality_condition_z_i} and \eqref{subeq:optimality_condition_z_fi}, respectively. This completes the proof of Theorem~\ref{theorem:converge_outer}.
\end{proof}

\begin{remark}
	Our algorithm explicitly demonstrates how to apply the hierarchical ADMM in the multi-agent system with the presence of nonconvex collision-avoidance constraints. The closed-form solution of every variable is provided. In this approach, a proximal term is introduced to the two-level algorithm presented in~\cite{sun2019two} and a more compact convergence condition is proposed, i.e., $\rho_i \geq \sqrt{2} \beta_i$. Therefore, the Theorem 1 in our paper proves the convergence of the inner iteration under the condition $\rho_i \geq \sqrt{2} \beta_i$. Theorem 2 is proved under different assumptions, compared with the algorithm proposed in~\cite{sun2019two}. Thus, we believe our algorithm and theorems provide more details about the hierarchical ADMM algorithms for the multi-gent systems.
\end{remark}

The stopping criterion of this ADMM algorithm is given by
\begin{IEEEeqnarray*}{l}
	\label{eq:stopping_criterion}
	\left\|\rho_i^k\mathcal A_i^T\left( \mathcal B_iz_{f,i}^{r} + s_i^{r}-\mathcal B_iz_{f,i}^{r+1} - s_i^{r+1}  \right)\right\| \leq \epsilon_1^k  \IEEEyesnumber \IEEEyessubnumber\\
	\left\|\rho_i^k\mathcal B_i^T\left(s_i^r-s_i^{r+1}\right)\right\| \leq \epsilon_2^k  \IEEEnonumber \IEEEyessubnumber \\
	\left\|\mathcal A_iz_i^{r+1} + \mathcal B_iz_{f,i}^{r+1} - h_i + s_i^{r+1} \right\| \leq \epsilon_3^k,  \IEEEnonumber \IEEEyessubnumber\textbf{}
\end{IEEEeqnarray*}
where $\epsilon_e^k\rightarrow 0, \forall e \in \{1,2,3\}$. 

To summarize the above developments, the pseudocode of this algorithm is presented in Algorithm~\ref{alg:hierarchicalADMM}. After obtaining each agent's control input and executing the control input, the agents will communicate with their neighbors to obtain the updated position information of their neighbors. Then, this process will be repeated until finishing the designed task.
\begin{algorithm}
	\caption{Hierarchical ADMM for Multi-Agent Decision Making}
	\label{alg:hierarchicalADMM}
	\begin{algorithmic}
		\STATE {Initialize $z_i^0\in \mathbb R^{T(m+n)}, z_{f,i}^0\in \mathbb R^{2T(m+n)}, s_i^0\in \mathbb R^{2Tm+3Tn}$; initialize the dual variable $\lambda_i^0\in [\underline{\lambda}_i, \overline{\lambda}_i]$, where $\underline{\lambda}_i, \overline{\lambda}_i \in \mathbb R^{2Tm+3Tn}$, and $\overline{\lambda}_i > \underline{\lambda}_i$; initialize the penalty parameter $\beta_i^0>0$, and the constant $\omega\in [0,1)$ and $\gamma>1$; initialize the sequence of tolerance $\{\epsilon_1^k, \epsilon_2^k,\epsilon_3^k\} \in \mathbb R_+$ with $\lim_{k\rightarrow\infty} \epsilon_e^k = 0, \forall e\in \{1,2,3\}$}. Given the terminated tolerances $\epsilon_{\text{term},e}, e = \{1,2,3\}$, if $\epsilon_e^k \leq \epsilon_{\text{term},e}$, $\epsilon_e^k = \epsilon_{\text{term},i}$ for $e = \{1,2,3\}$.
		\STATE {Set the outer iteration $k = 0$.}
		\WHILE {Stopping criterion is violated}
		\STATE {Given $\lambda_i^k$ and $\beta_i^k$, initialize $z_i^{k,0}, z_{f,i}^{k,0}, s_i^{k,0}$ and $y_i^{k,0}$ such that $\lambda_i^k+\beta_i^ks_i^{k,0}+y_i^{k,0}=0$; initialize tolerance $(\epsilon_1^k, \epsilon_2^k, \epsilon_3^k)$; $\rho_i^k \geq \sqrt{2}\beta_i^k$.}
		\STATE {Set the inner iteration $r=0$.}
		\WHILE {Stopping criterion is violated}
		\STATE {Update $z_i^{k,r+1}$ by~\eqref{eq:update_z_i} via nonlinear programming solvers.}
		\STATE {Update $z_{f,i}^{k,r+1}$, $s_i^{k,r+1}$, $y_i^{k,r+1}$ by~\eqref{eq:update_z_f_i},~\eqref{eq:update_s_i},~\eqref{eq:update_y_i}, respectively.}
		\STATE {$r = r+1$.}
		\ENDWHILE
		\STATE {Set $z_i^{k+1}, z_{f,i}^{k+1}, s_i^{k+1}$, $y_i^{k+1}$ to be $z_i^{k,r+1}, z_{f,i}^{k,r+1}, s_i^{k,r+1}$, $y_i^{k,r+1}$, respectively.}
		\STATE {Update $\lambda_i^{k+1}$ and $\beta_i^{k+1}$ by~\eqref{eq:update_lambda_beta}.}
		\STATE {$k=k+1$.}
		\ENDWHILE
	\end{algorithmic}
\end{algorithm}

\section{Improved Hierarchical ADMM}
\label{section:improved_Hierarchical_ADMM}
In the hierarchical ADMM approach, the minimization of $z_i$ in each iteration is computationally expensive, and it requires nonlinear programming (NLP) solvers. Therefore, the improved minimization of $z_i$ is required to improve the computational efficiency by approximating and solving this inexact minimization problem. 

The original $z_i$-minimization problem can be expressed as
\begin{IEEEeqnarray*}{rCl}
\label{eq:z_i_minimization}
\underset{z_i}{\operatorname{min}} &\quad& J_i(z_i) + \frac{\rho_i^k}{2} \left\| \mathcal A_iz_i + \mathcal B_iz_{f,i}^{k,r} + s_i^{k,r} - h_i+\frac{y_i^{k,r}}{\rho_i^{k}} \right\|^2 \\
\operatorname{s.t.} && \phi\left(z_i; [p_j]\right)\leq 0 \\
&& \psi\left(z_i; [p_j]\right) \leq 0 , \yesnumber
\end{IEEEeqnarray*}
where
\begin{IEEEeqnarray*}{rCl}
	\phi\left(z_i; [p_j]\right) &=& \mathcal{M}\left(z_i; [p_j]\right) \mathcal{P}\left(z_i; [p_j]\right) - d_{\mathrm{cmu}}^2\\
	\psi\left(z_i; [p_j]\right)	&=& d_{\mathrm{safe}}^2 - \mathcal{M}\left(z_i;[p_j]\right) \mathcal{P}\left(z_i; [p_j]\right).
\end{IEEEeqnarray*}
The KKT conditions of this $z_i$-minimization problem~\eqref{eq:z_i_minimization} are
\begin{IEEEeqnarray*}{rCl}
	0 &=& \nabla J_i(z_i) + {\rho_i^k} \mathcal A_i^\top \left( \mathcal A_iz_i + \mathcal B_iz_{f,i}^{k,r} + s_i^{k,r} - h_i+\frac{y_i^{k,r}}{\rho_i^{k}} \right) \\
	  & & + \sum\limits_{\theta=1}^{\theta_\phi} \mu_{\theta,i} \nabla \phi_{\theta} \left(z_i; [p_j]\right) + \sum\limits_{\theta=1}^{\theta_\psi} \kappa_{\theta,i} \nabla \psi_{\theta}\left(z_i; [p_j] \right) \\
	0 &=& \mu_{\theta,i} \phi_{\theta}\left(z_i;[p_j]\right), \: \forall \theta = 1,2,\cdots, \theta_\phi \\
	0 &=& \kappa_{\theta,i} \psi_{\theta}\left(z_i;[p_j]\right) , \: \forall \theta = 1,2,\cdots, \theta_\psi. \yesnumber
\end{IEEEeqnarray*} 

Notice that the hierarchical ADMM algorithm requires complete minimization of $z_i$ in each inner iteration. However, this is neither desirable due to the computational cost nor practical, as any NLP solver finds only a point that approximately satisfies the KKT optimality conditions. In the hierarchical ADMM algorithm, we propose the use of the interior-point method to solve the $z_i$ minimization problem~\eqref{eq:z_i_minimization}. The interior-point method employs double-layer iterations to find the solution. In the outer iteration, a barrier technique is used to convert the inequality constraints into an additional term in the objective; the stationary points of resulting barrier problems converge to true stationary points as the barrier parameter converges to 0. In the inner iteration, a proper search method is used to obtain the optimum of the barrier problem. Since both the interior-point algorithm and the hierarchical ADMM algorithm have a double-layer structure, we consider matching these two layers.

In order to accelerate the computation process of solving~\eqref{eq:z_i_minimization}, the barrier method can be employed to transform the inequality constraints to additional terms in the objective function. For example, in the $k$-th outer loop, the function $J_i(z_i)$ can be added with a barrier term $-b_i^k \sum_{\theta_{\phi}} \ln(-\phi(z_i))$ based on the barrier method, where $b_i^k$ is the barrier parameter with $\lim\limits_{k \rightarrow \infty} b_i^k = 0$. Hence, the barrier augmented Lagrangian in the $k$-th outer loop can be written as 
\begin{IEEEeqnarray}{rCl}
\label{eq:barrier_L}
{\hat{\mathcal L}}_{b_i^k} = \mathcal L_{\rho_i^k} - b_i^k\sum\limits_{\theta=1}^{\theta_\phi} \ln(-\phi_\theta(z_i)).
\end{IEEEeqnarray}
If the $z_i$-minimization step returns $z_i^{k,r+1}$ by minimizing $\hat{\mathcal L}_{b_i}$ with respect to $z_i$, then the inner iterations lead to the decrease of $\hat{\mathcal L}_{b^k_i}$, which also matches the KKT conditions of the original optimization problem. Therefore, the outer iterations can find an approximate stationary point of the original problem with the decrease of the barrier parameter $b_i^k$. Note that only $\psi(z_i; [p_j])$ is used as the inequality constraint for the purpose of demonstration.

Therefore, it remains to minimize $\hat{\mathcal L}_{b_i}$, which is given by
\begin{IEEEeqnarray*}{rCl}
\hat{\mathcal L} _{b_i^k}
 &=& J_i(z_i) + \frac{\rho_i^k}{2} \left\| \mathcal A_iz_i + \mathcal B_iz_{f,i}^{k,r} + s_i^{k,r} - h_i+\frac{y_i^{k,r}}{\rho_i^{k}} \right\|^2 \\
 && - \left\langle b_i^k 1_{Tr_i}, \ln \left(h\left(z_i; [p_j]\right) - d_{\mathrm{safe}}^2 \right) \right\rangle, \yesnumber
\end{IEEEeqnarray*}
where $	h\left(z_i; [p_j]\right) = {\mathcal{M}\left(z_i; [p_j]\right) \mathcal{P}\left(z_i; [p_j]\right)}$.

\begin{lemma}~\label{lemma:majun}
	The gradient of
	$h\left(z_i; [p_j]\right)$ with respect to $z_i$ is a linear mapping with respect to $z_i$.
\end{lemma}
\begin{proof}
	Define $W = I_{Tr_i}\otimes 1_{n_p}\in\mathbb R^{Tr_in_p\times Tr_i}$, $\omega = (I_T\otimes1_{r_i}\otimes I_{n_p})M_pz_i - 1_{T} \otimes [p_j] \in \mathbb R^{Tr_in_p}$, $\Omega = \operatorname{diag}(\omega)\in \mathbb R^{Tr_in_p\times Tr_in_p}$. Then, 
	\begin{IEEEeqnarray*}{rCl}
		h\left(z_i; [p_j]\right) = W^\top(I_{Tr_in_p}\odot(\omega 1_{Tr_in_p}^\top)) \omega = W^\top \Omega \omega. 
	\end{IEEEeqnarray*}
	Since $\Omega = I_{Tr_in_p}\odot(\omega 1_{Tr_in_p}^\top) = I_{Tr_in_p} \odot(\mathbf1_{Tr_in_p} \omega^\top)$, we have
	\begin{IEEEeqnarray*}{rCl}
		\textup{d} h\left(z_i;[p_j]\right) 
		&=& 2 W^\top \Omega (I_T\otimes 1_{r_i}\otimes I_{n_p})M_p \textup{d}z_i. \IEEEeqnarraynumspace \yesnumber
	\end{IEEEeqnarray*}
	It follows 
	\begin{IEEEeqnarray*}{rCl}
		&&  \frac{\textup{d} h\left(z_i;[p_j]\right)}{ \textup{d} z_i} \\
		&=& 2(I_{Tr_i}\otimes 1_{n_p})^\top \Big[ I_{Tr_in_p} \odot \Big( \big((I_T\otimes 1_{r_i}\otimes I_{n_p})M_pz_i \\
		 && \; - 1_{T\times 1} \otimes [p_j]\big)  1_{Tr_in_p}^T \Big) \Big] \cdot  (I_T\otimes 1_{r_i}\otimes I_{n_p})M_p .  \yesnumber
	\end{IEEEeqnarray*}
Therefore, Lemma~\ref{lemma:majun} is proved.
\end{proof}

Set $\Upsilon_i = -\psi \left(z_i; [p_j]\right)$ and define $\Xi_i$ such that $\Xi_i \odot \Upsilon_i = 1_{Tr_i}$. 
Then, the gradient of the barriered augmented Lagrangian function $\hat{\mathcal L} _{b_i^k}$ is 
\begin{IEEEeqnarray*}{rCl}
	\label{eq:gradient_L_b}
   && \nabla \hat{\mathcal L} _{b_i^k} \\
  &=& 2H_i(z_i-z_{i,ref}) + {\rho_i^k} \mathcal A_i^\top \Bigg( \mathcal A_iz_i + \mathcal B_iz_{f,i}^{k,r} + s_i^{k,r} - h_i \\
   && +\frac{y_i^{k,r}}{\rho_i^{k}} \Bigg) - 2b_i^k \operatorname{Tr} \left(1_{Tr_i}^\top \Xi_i \odot W^\top \Omega (I_T\otimes 1_{r_i}\otimes I_{n_p}) M_p \right) . \\ \yesnumber 
\end{IEEEeqnarray*}

A closed-form algebraic solution of problem~\eqref{eq:z_i_minimization} is challenging to determine. However, since the gradient $\nabla \hat{\mathcal L} _{b_i^k}$ is available, some well-known gradient-based numerical methods, e.g., Barzilai-Borwein, Polak-Ribiere, and Fletcher-Reeves, can be used to calculate the optimal $z_i$. Therefore, the $z_i$-update step in Algorithm~\ref{alg:hierarchicalADMM} can be replaced by solving an unconstrained optimization problem~\eqref{eq:z_i_minimization} with the known gradient~\eqref{eq:gradient_L_b}. 

Usually, iterative algorithms for NLP need to be called when solving the $z_i$ minimization problem~\eqref{eq:z_i_minimization}, and always searching for a highly accurate solution in each ADMM iteration will result in an excessive computational cost. It is thus desirable to solve the optimization subproblem~\eqref{eq:z_i_minimization} in ADMM inexactly when the dual variables are yet far from the optimum, i.e., to allow $z_i^{k,r+1}$ to be chosen such that
\begin{IEEEeqnarray}{rCl}
	\label{eq:error_d}
	d_x^{k+1} \in \partial_{z_i} \mathcal{L}_{\rho_i}.
\end{IEEEeqnarray}
In this paper, we use externally shrinking and summable sequence of the absolute errors, i.e.,
\begin{IEEEeqnarray}{rCl}
	\label{eq:error_d_bound}
	\|d_z^k\| \leq \epsilon_4^k, \sum_{k=1}^{\infty} \epsilon_4^k \leq \infty.
\end{IEEEeqnarray}
such as a summable absolute error sequence can effectively reduce the total number of subroutine iterations throughout the ADMM algorithm. Such a criterion is a constructive one, rendered to guarantee the decrease of a quadratic distance between the intermediate solutions $(z_i^k, z_{f,i}^k, s_i^{k}, y_i^k)$ and the optimum $(z_i^*, z_{f,i}^k, s_i^{k}, y_i^k)$.

The above approximate NLP solution is realizable by NLP solvers where the tolerances of the KKT conditions are allowed to be specified by the user, e.g., the IPOPT solver. The pseudocode of the detailed algorithm is shown in Algorithm~\ref{alg:improved_hierarchicalADMM}.

\begin{algorithm}
	\caption{Improved Hierarchical ADMM for Multi-Agent Decision Making}
	\label{alg:improved_hierarchicalADMM}
	\begin{algorithmic}
		\STATE {Initialize $z_i^0\in \mathbb R^{T(m+n)}, z_{f,i}^0\in \mathbb R^{2T(m+n)}, s_i^0\in \mathbb R^{2Tm+3Tn}$; initialize the dual variable $\lambda_i^0\in [\underline{\lambda}_i, \overline{\lambda}_i]$, where $\underline{\lambda}_i, \overline{\lambda}_i \in \mathbb R^{2Tm+3Tn}$, and $\overline{\lambda}_i > \underline{\lambda}_i$; initialize the penalty parameter $\beta_i^0>0$, and the constant $\omega\in [0,1)$ and $\gamma>1$; initialize the sequence of tolerance $\{\epsilon_1^k, \epsilon_2^k,\epsilon_3^k, epsilon_4^k, \epsilon_b\} \in \mathbb R_+$ with $\lim_{k\rightarrow\infty} \epsilon_e^k = 0, \forall e\in \{1,2,3,4,b\}$}. Given the terminated tolerances $\epsilon_{\text{term},e}, e = \{1,2,3,4,b\}$, if $\epsilon_e^k \leq \epsilon_{\text{term},e}$, $\epsilon_e^k = \epsilon_{\text{term},i}$ for $e = \{1,2,3,4,b\}$.
		\STATE {Set the outer iteration $k = 0$.}
		\WHILE {Stopping criterion is violated \textbf{or} $b_i^k \geq \epsilon_b$}
		\STATE {Given $\lambda_i^k$ and $\beta_i^k$, initialize $z_i^{k,0}, z_{f,i}^{k,0}, s_i^{k,0}$ and $y_i^{k,0}$ such that $\lambda_i^k+\beta_i^ks_i^{k,0}+y_i^{k,0}=0$; initialize tolerance $(\epsilon_1^k, \epsilon_2^k, \epsilon_3^k)$; $\rho_i^k \geq \sqrt{2}\beta_i^k$.}
		\STATE {Set the inner iteration $r=0$.}
		\WHILE {Stopping criterion is violated}
		\STATE {Update $z_i^{k,r+1}$ by minimize~\eqref{eq:barrier_L}.}
		\STATE {Update $z_{f,i}^{k,r+1}$, $s_i^{k,r+1}$, $y_i^{k,r+1}$ by~\eqref{eq:update_z_f_i},~\eqref{eq:update_s_i},~\eqref{eq:update_y_i}, respectively.}
		\STATE {$r = r+1$.}
		\ENDWHILE
		\STATE {Set $z_i^{k+1}, z_{f,i}^{k+1}, s_i^{k+1}$, $y_i^{k+1}$ to be $z_i^{k,r+1}, z_{f,i}^{k,r+1}, s_i^{k,r+1}$, $y_i^{k,r+1}$, respectively.}
		\STATE {Update $\lambda_i^{k+1}$ and $\beta_i^{k+1}$ by~\eqref{eq:update_lambda_beta}.}
		\STATE {$k=k+1$.}
		\ENDWHILE
	\end{algorithmic}
\end{algorithm}

In this method, a nonlinear programming problem without constraints is solved with inequality constraints handled by a fixed barrier constant throughout inner iterations. The barrier constant decreases across outer iterations. In the hierarchical ADMM, the NLP solver needs to solve an inequality-constrained NLP in each inner iteration. When the NLP solver, i.e., interior-point algorithm, is called for NLP with inequalities, the barrier constant will go through an iteration inside the solver, and thus every inner iteration will have more loops of internal iteration. By this barrier Lagrangian construction, we integrate the inner loop of the IPOPT with the inner loop of the hierarchical ADMM algorithm. In addition, in the improved hierarchical ADMM, each inner iteration only solves NLP, and an approximate solution is reached without the need of convergence. The accuracy, namely the tolerances of the errors from the KKT conditions, can be controlled when using the interior-point algorithm. Therefore, these tolerances can be set adaptively throughout the iterations (for instance, the tolerances of the 1st, 5th, and 20th iteration can be set to 0.1, 0.01, and 0.001, respectively). By this adaptive tolerance setting, we integrate the outer loop of the IPOPT with the inner loop of the hierarchical ADMM. Therefore, the improved hierarchical ADMM gives less iteration number and computation time than the hierarchical ADMM.

\section{Illustrative Example}
\label{section:Illustrative_Examples}
In order to show the effectiveness of the hierarchical ADMM and improved hierarchical ADMM, a multi-UAV system is presented as an application test platform for demonstration of the proposed decision-making framework. The simulations are performed in Python 3.7 environment with processor Intel Xeon(R) E5-1650 v4 @ 3.60GHz. All nonlinear problems are solved by the open source IPOPT solver.  
\subsection{Dynamic Model}
For each quadcopter system, the system model can be expressed in the following state-space representation:
\begin{IEEEeqnarray*}{c}
	\label{eq:model}
	\dot{x} = {\begin{bmatrix} {\dot{ p}} \\ {\dot{ v}} \\ {\dot{\zeta}} \\ {\dot{\omega}}\end{bmatrix}} =
	{\begin{bmatrix} { R(\phi, \theta, \psi)}^\top { v} \\ {-{\omega}\times  v + g { R(\phi, \theta, \psi)}  e} \\  W(\phi, \theta, \psi) {\omega} \\  J^{-1}(-{\omega} \times { J} {\omega}) \end{bmatrix}
		+ {\tilde{ B}} {({ u}_{\mathrm{eq}} + { u})}}, \IEEEeqnarraynumspace \yesnumber 
\end{IEEEeqnarray*}
where $x\in\mathbb{R}^{12}$ and $u\in\mathbb{R}^4$ are state and control input vector for the agents, respectively. In this dynamics model, rotor thrusts of each rotor are chosen as control input vector, i.e., $u=\begin{bmatrix}F_1 & F_2 & F_3 & F_4\end{bmatrix}^\top$. ${p}=\begin{bmatrix} p_x & p_y & p_z \end{bmatrix}^\top$ is a vector comprising the position  in the $x$-, $y$-, and $z$-dimension. Additionally, the velocity vector is represented by ${v}=\begin{bmatrix} v_x & v_y & v_z\end{bmatrix}^\top$. $\zeta=\begin{bmatrix} \phi& \theta & \psi \end{bmatrix}^\top$ denotes a vector in terms of the roll, pitch, and yaw angles. The angular velocity vector is represented by $\omega=\begin{bmatrix} \omega_x & \omega_y & \omega_z\end{bmatrix}^\top$. Also, ${e}=\begin{bmatrix} 0 & 0 &1\end{bmatrix}^\top$, $ u_{\mathrm{eq}} = \begin{bmatrix}  {mg}/{4} &  {mg}/{4} &  {mg}/{4} &  {mg}/{4} \end{bmatrix}^\top$ is the control input to balance the gravity of the agents, $g$ is the gravitational acceleration, $m$ is the mass of the quadcopter, $ J=\textup{diag}(J_x, J_y, J_z) $ denotes the moment of inertia of the quadcopter, $R(\phi,\theta,\psi)$ and $W(\phi,\theta,\psi)$ denote the rotation matrices of the quadcopter. In addition, ${\tilde{B}}$ is the control input matrix in the above state-space representation. 

Here,~\eqref{eq:model} is nonlinear and time-variant, which can be represented as a pseudo-linear form by using state-dependent coefficient factorization~\cite{zhang2019integrated}. The state-space expression is 
\begin{IEEEeqnarray}{rCl}
	\label{eq:dynamics_state-space}
	{x_{t+1}}=(\tilde{ A}_t\Delta t + I_{12})  x_{t} + \tilde{ B} \Delta t ( u_\textup{eq}+  u_{t}), \IEEEeqnarraynumspace
\end{IEEEeqnarray}
where $\Delta t$ is the sampling time interval. Since ${\tilde{A}}_t$ and $\tilde{B}$ are dependent on the current state $x$, this state-space representation is in a pseudo-linear form, and then we can suitably consider the system matrices to be constant during the prediction horizon. 
More details of the parameter settings in \eqref{eq:model} and \eqref{eq:dynamics_state-space} are presented in~\cite{zhang2020trajectory}. 

\subsection{Optimization and Simulation Results}
The dynamic model of each agent in the multi-UAV system is shown in~\eqref{eq:dynamics_state-space}. All the agents are set evenly distributed in a circle with the radius of $2$ m and $p_z=0$ m, and each agent needs to reach its corresponding point in this same sphere. Therefore, a large number of potential collision possibilities are set up in the simulation task, and our objective is to avoid these designed potential collisions among these agents and make all agents reach their desired destinations. Here, a distributed MPC with a quadratic objective function is designed to realize the path tracking and collision avoidance tasks. Set the number of agents $N=8$. The safety distance $d_{\text{safe}}$ is $0.2$ m. The dimensions are $n=12, m=4,n_p=3$. In addition, the agents are connected with all of the remaining agents. The adjacency matrix of the communication network $\mathcal D$ is the difference between the upper triangular matrix of the square one matrix $1_{(N,N)}$ and the identity matrix $I_N$, i.e., $\mathcal D = U(1_{(N,N)}) - I_N$.  In this adjacency matrix $\mathcal D$, each entry $\mathcal D_{ij}$ represents whether there is the connection between the $i$th agent and the $j$th agent. If $\mathcal D_{ij} = 1$, there is a communication connection between the two agents; otherwise, there is no connection between the two agents. For example, the $i$th agent will get communication with the $i+1, i+2, \cdots, N$ agents.
The velocity of the agents is confined into [-5, 5] m/s, and the upper and lower bounds of angles $(\phi,\psi, \theta)$ in three dimensions are $(\pi, \frac{\pi}{2}, \pi)$ and $(-\pi, -\frac{\pi}{2}, -\pi)$. The upper bound of control input variable is $\overline u = (1.96, 1.96, 1.96, 1.96)$ N and its lower bound is $\underline u = -\overline u$. 
Set the prediction horizon $T=25$ with the sampling time $\Delta t=0.05$ s. The weighting matrices $R_i = 0_{(m,m)}$ and $Q_i = \operatorname{blockdiag}(I_{n_p}, 0_{(n-n_p,n-n_p)})$ for all $i\in \mathcal V$. Each element in $\overline \lambda_i$ and $\underline \lambda_i$ are set as $0.01$ and $-0.01$, respectively. $\omega$ and $\gamma$ are set as 0.9 and 1.1. The initial tolerances $\epsilon_1^0, \epsilon_2^0, \epsilon_3^0$ are defined as $10^{-2}, 10^{-2}, 10^{-1},$ respectively, and $\epsilon_i^k = \epsilon_i^{k-1}/5^{k-1}, \forall i \in \{1,2,3\}$. The terminated tolerances for outer iterations $\epsilon_{\text{term},1}$, $\epsilon_{\text{term},2}$, $\epsilon_{\text{term},3}$ are $10^{-6},10^{-6},10^{-5}$, respectively. 
The maximum iteration number is set to $10^4$.

The reference trajectories and the resulted trajectories computed by the improved hierarchical ADMM are shown in Fig.~\ref{fig:HierADMM_traj} and Fig.~\ref{fig:2d_view}. The initial positions of the 8 agents are distributed uniformly in a circle, whose radius equals 2 m, and the z-axis is 0. In this simulation task, the 8 agents need to move towards their corresponding points, which are the central symmetric points of their initial position, respectively. For example, the agents with initial position (0,2,0) m and ($\sqrt{2},\sqrt{2},0$) m needs to reach (0,-2,0) m and ($-\sqrt{2},-\sqrt{2},0$) m.  In the two figures, each agent is represented by a different color. These lines in different colors denote the resulted trajectories of the 8 agents, respectively. From the two figures, it is obvious that all agents can effectively avoid their neighbors and maintain the safety distance. On the other hand, for the hierarchical ADMM, a similar performance of collision avoidance can be attained.

\begin{figure}[t]
	\centering
	\includegraphics[width=0.4\textwidth, trim=10 10 0 25,clip]{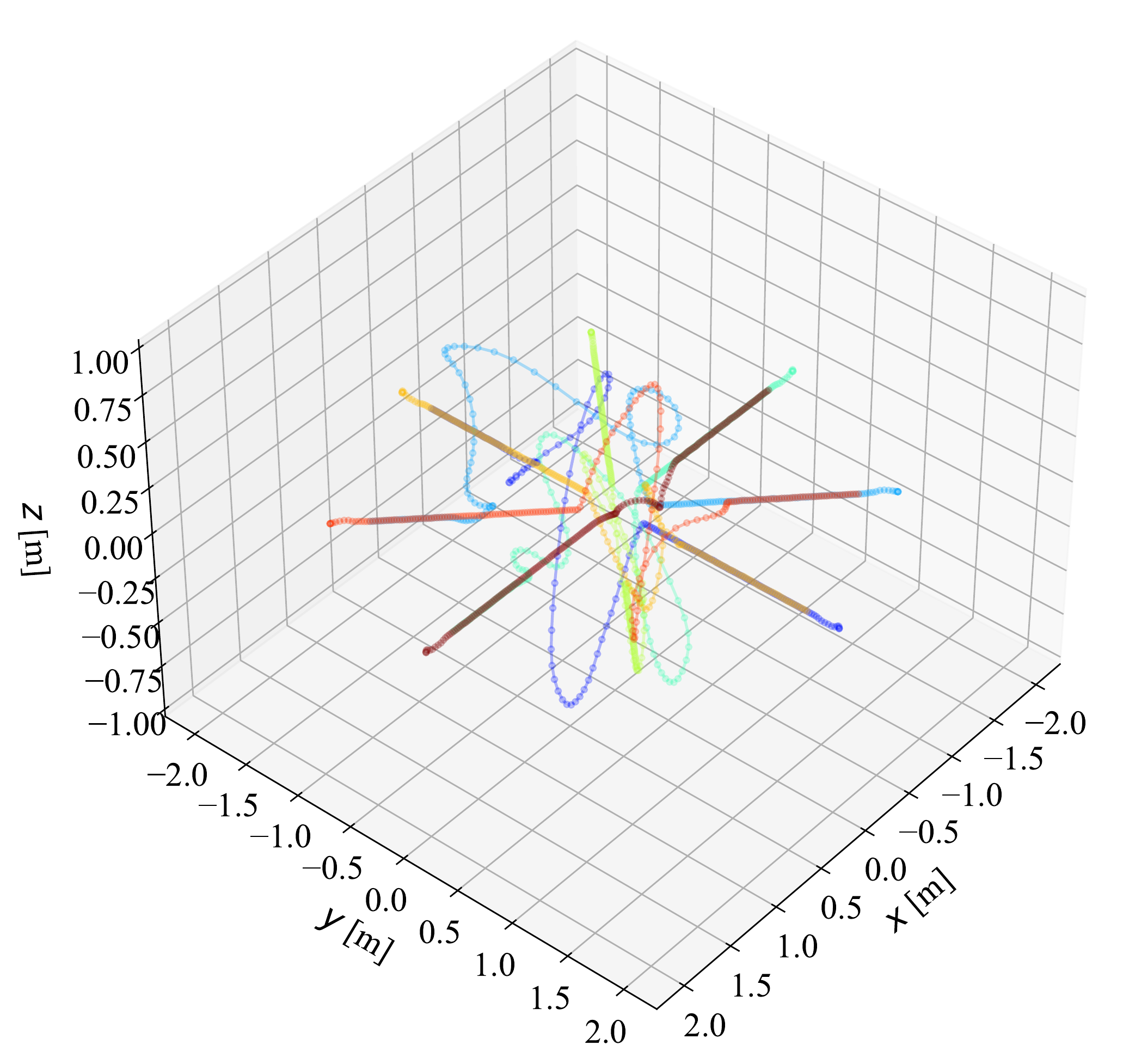}
	\caption{\label{fig:HierADMM_traj} Resulted trajectories of 8 agents in 3D view.}
\end{figure}
\begin{figure}[t]
	\centering
	\includegraphics[width=0.45\textwidth, trim=0 10 0 8,clip]{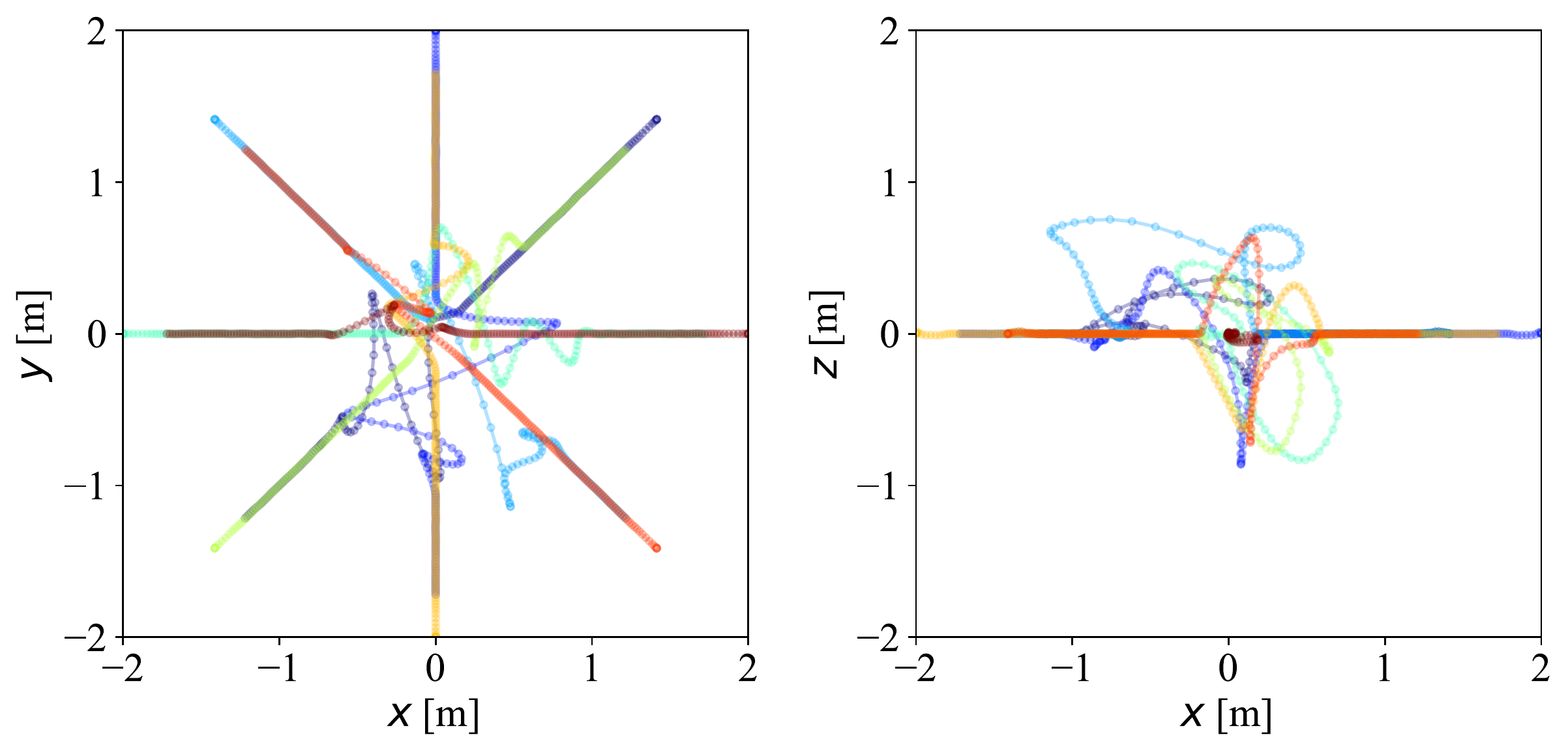}
	\caption{\label{fig:2d_view} Resulted trajectories of 8 agents in 2D view.}
\end{figure}

In order to show the bound of the control inputs, we only choose the control input results of the 1st agent $u = \begin{bmatrix} F_1 & F_2 & F_3 & F_4 \end{bmatrix}$ as an example for the clarity of this figure, as shown in Fig.~\ref{fig:HierADMM_inputs}. The 4 dimensions of the control inputs are bounded in the predefined range $[-1.96, 1.96]$ N.

\begin{figure}[!t]
	\centering
	\includegraphics[width=0.45\textwidth, trim=0 10 0 7,clip]{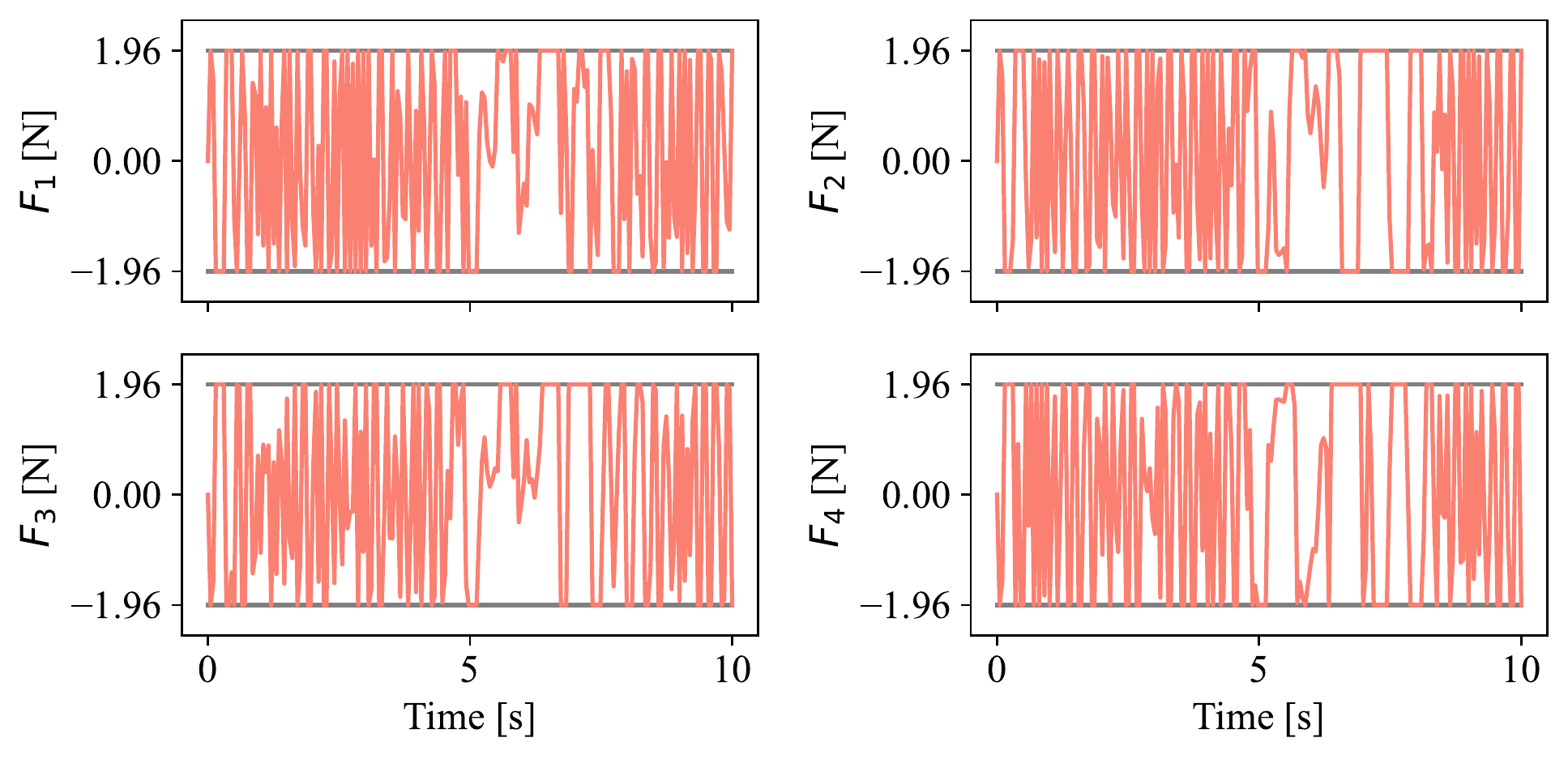}
	\caption{\label{fig:HierADMM_inputs}Resulted control inputs of the 1st agent.}
\end{figure}

The distances between each pair of the 8 agents are shown in Fig.~\ref{fig:HierADMM_distance}. The solid gray line in this figure represents the safety distance all of the agents need to maintain, and the other colored lines denote the distance of each pair of agents with respect to time. It shows that the distance between agents are no smaller than the safety distance.
\begin{figure}[!t]
	\centering
	\includegraphics[width=0.35\textwidth, trim=0 10 0 8,clip]{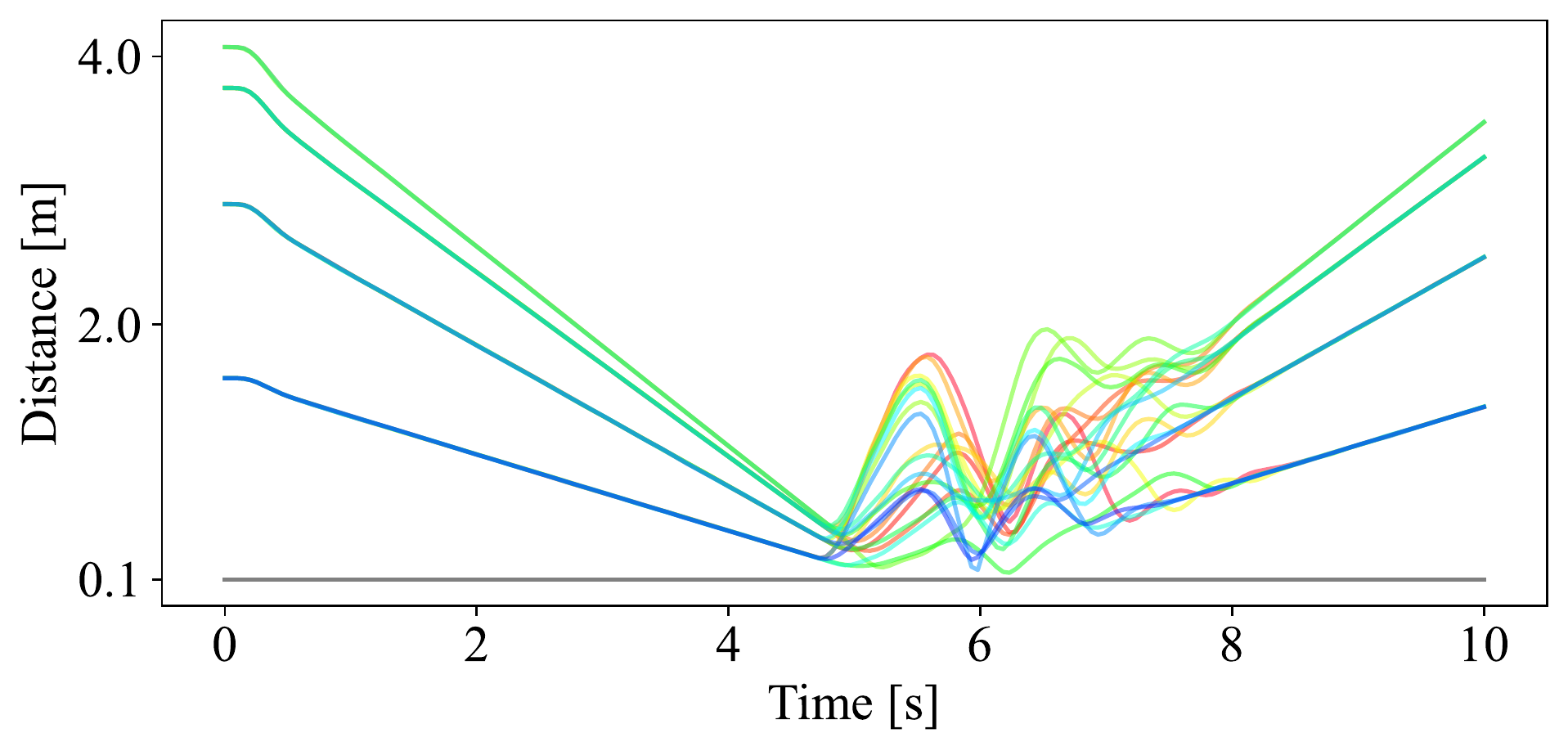}
	\caption{\label{fig:HierADMM_distance} Distances between each pair of agents.}
\end{figure}
The comparison of computational efficiency between our proposed hierarchical ADMM and improved hierarchical ADMM can be illustrated, according to Fig.~\ref{fig:compare_iters} and Fig.~\ref{fig:compare_comp_time}. Fig.~\ref{fig:compare_iters} demonstrates the comparison of iteration numbers the inner loop and outer loop require in both hierarchical ADMM and improved hierarchical ADMM. It is straightforward to observe that the improved hierarchical ADMM can effectively reduce the iteration numbers of both the inner loop and outer loop, compared with the hierarchical ADMM. The comparison of computation time for all inner loop iterations between hierarchical ADMM and improved hierarchical ADMM is shown in Fig.~\ref{fig:compare_comp_time}. From this figure, it can be observed that the computation time is significantly reduced by using the improved version. 
\begin{figure}[!t]
	\centering
	\includegraphics[width=0.35\textwidth, trim=0 15 0 8,clip]{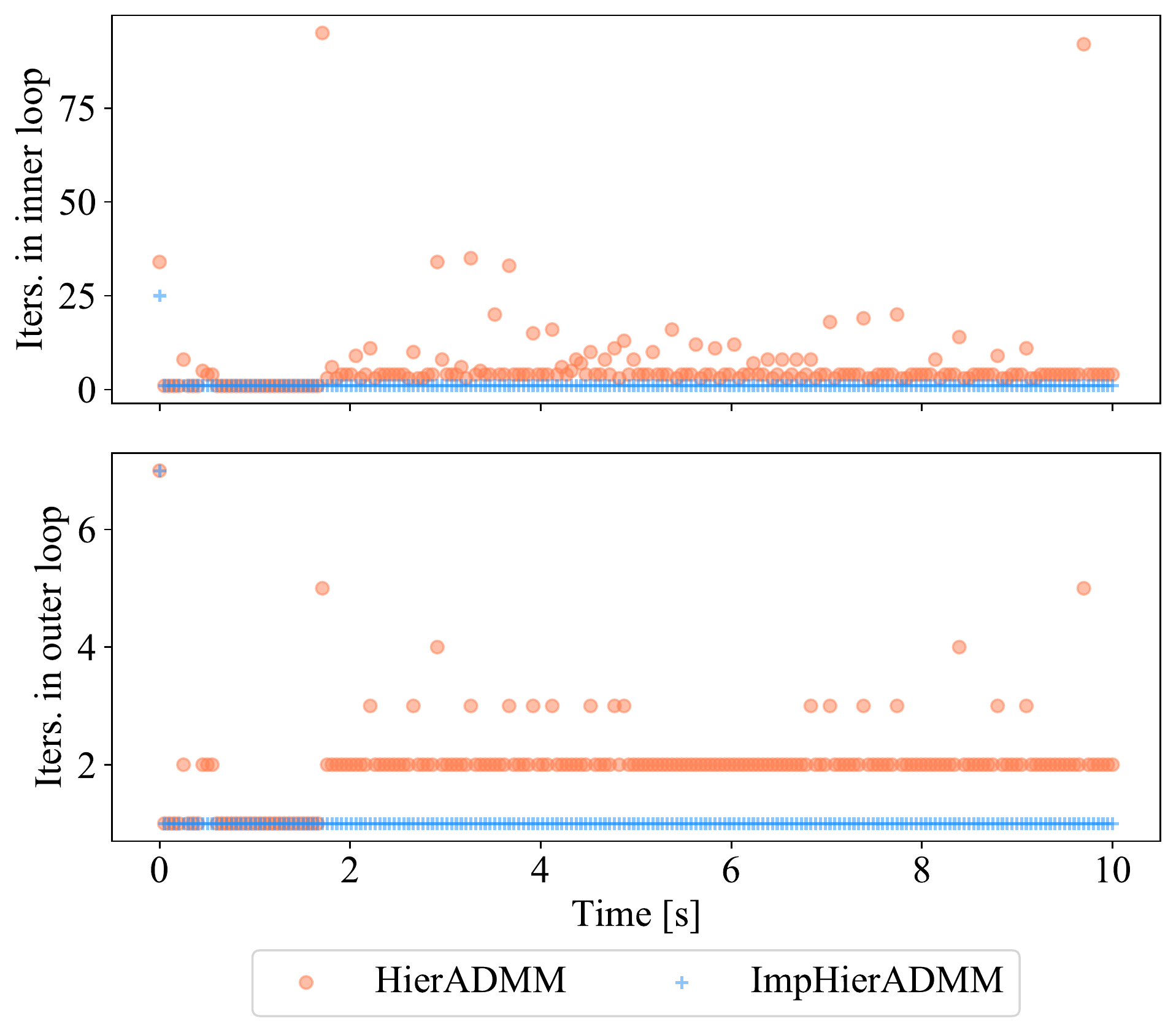}
	\caption{\label{fig:compare_iters} Comparison of inner loop and outer loop iteration numbers between hierarchical ADMM and improved hierarchical ADMM.}
\end{figure}

\begin{figure}[!t]
	\centering
	\includegraphics[width=0.35\textwidth, trim=0 15 0 7,clip]{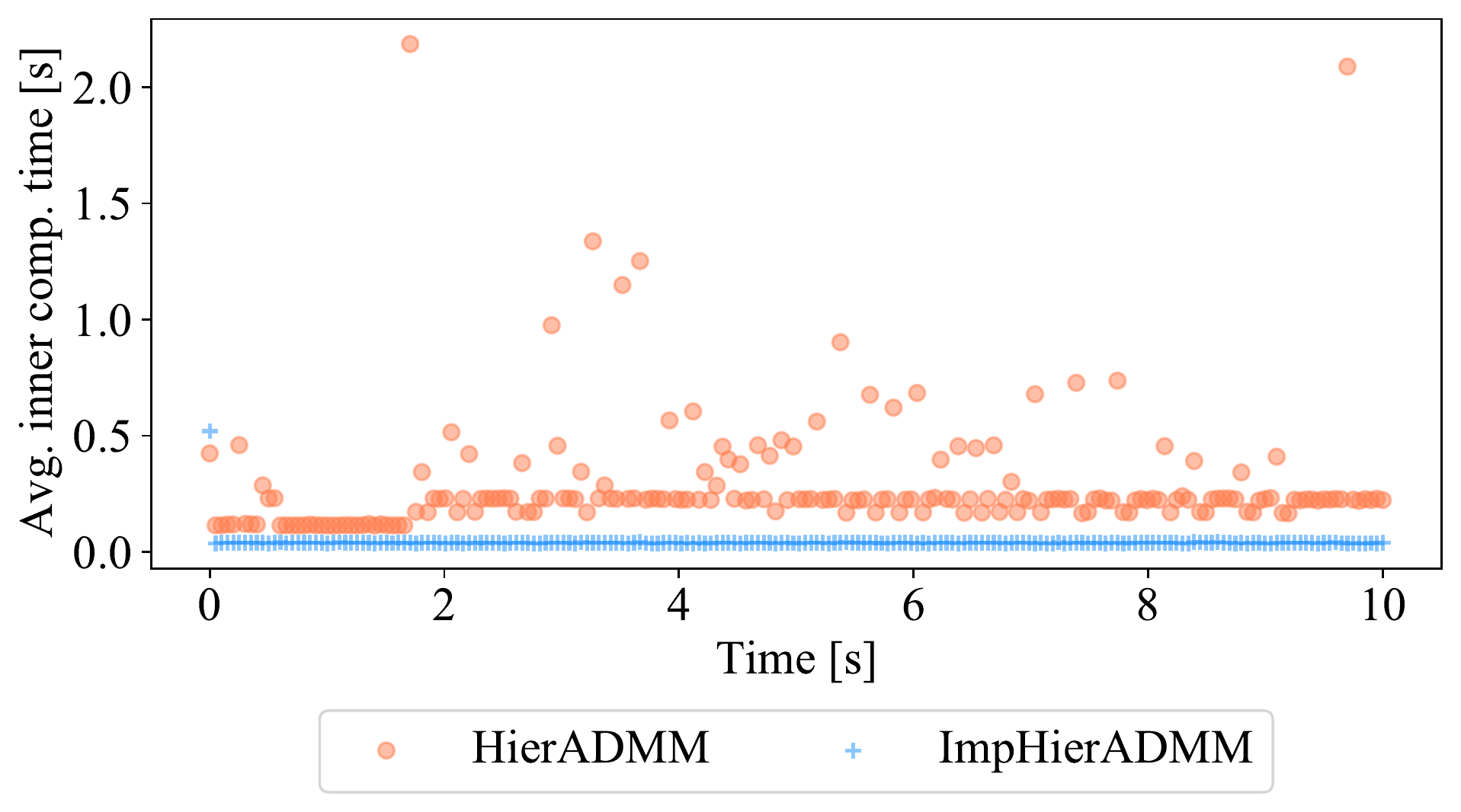}
	\caption{\label{fig:compare_comp_time} Comparison of computation time of all inner loop iterations between hierarchical ADMM and improved hierarchical ADMM.}
\end{figure}

\begin{table*}[thbp]
	\begin{center}
		\caption{\label{table:comp}Comparison of proposed algorithms (hierarchical ADMM and improved hierarchical ADMM) with the centralized MPC and penalty dual decomposition method~\cite{shi2020penalty}}
		\begin{tabular}{|c|c|c|c|c|}
			\hline
			& Centralized MPC      & Penalty Dual Decomposition~\cite{shi2020penalty}  & Hierarchical ADMM & Improved Hierarchical ADMM  \\ \hline
			Average outer iterations                             & -                    & 4.26                                              & 2.12                                             & 1.03                                              \\ \hline
			Average inner iterations                             & -                    & 17.36                                             & 6.38                                             & 1.12                                              \\ \hline
			$\|\mathcal A_i z_i + \mathcal B_i z_{f,i} - h_i \|$ & 2.21E{-5} & 3.32E{-5}                              & 2.54E{-5}                             & 1.47E{-5}                              \\ \hline
			Cost value                                           & 253.12               & 315.35                                            & 294.23                                           & 245.36                                            \\ \hline
			Stopping criterions                                  & -                    & 1E-6, 7.52E{-2}, 2.32E{-3} & 1E{-6}, 9.38E{-2}, 3.5E{-3} & 1E{-6}, 3.28E{-2}, 2.64E{-3} \\ \hline
			Time                                                 & 1.47s                & 0.22s                                             & 0.16s                                            & 0.09s                                             \\ \hline
		\end{tabular}
	\end{center}
\end{table*}

In order to show the effectiveness of the proposed improved hierarchical ADMM, the penalty dual decomposition method~\cite{shi2020penalty} and the centralized MPC method are used as the comparison methods. The centralized MPC method means we solve the centralized MPC problem by using a NLP solver, and the penalty dual decomposition method is also a double-loop iterative algorithm~\cite{shi2020penalty}. The inner loop of the penalty dual decomposition method is used to solve a nonconvex nonsmooth augmented Lagrangian problem via block-coordinate-descent-type methods, and the outer loop of this algorithm is to update the dual variables and/or the penalty parameter~\cite{shi2020penalty}. Table~\ref{table:comp} illustrates the advantages of our proposed algorithms (hierarchical ADMM and improved hierarchical ADMM) over the penalty dual decomposition method and the centralized MPC method.

In order to demonstrate the performance in terms of computational efficiency, the comparison of computation time per step (including all outer loop iterations and inner loop iterations) under the hierarchical ADMM and the improved hierarchical ADMM, and the computation time per step under the centralized MPC (solved by the IPOPT solver) are shown in Fig.~\ref{fig:comptime_compare}. From this figure, the computational efficiency of the ADMM for solving large-scale problems can be clearly observed. Moreover, the improved hierarchical ADMM is capable of further improving the efficiency, as compared with the hierarchical ADMM. 
With the increasing number of agents, the number of nonlinear constraints increases rapidly, which leads to difficulty in solving the problem. Therefore, the computation time is increasing with the increase in the number of agents $N$, as shown in Fig.~\ref{fig:comptime_vs_numagents}. However, when the centralized MPC solver is used, the growth rate of the computation time is much higher than the ones when the hierarchical ADMM and improved hierarchical ADMM are used, as shown in  Fig.~\ref{fig:comptime_vs_numagents}, which also demonstrates the efficiency of the two proposed methods. Besides, the efficiency of the improved hierarchical ADMM can also be proved in this figure. 

\begin{figure}[!t]
	\centering
	\includegraphics[width=0.37\textwidth, trim=0 15 0 7,clip]{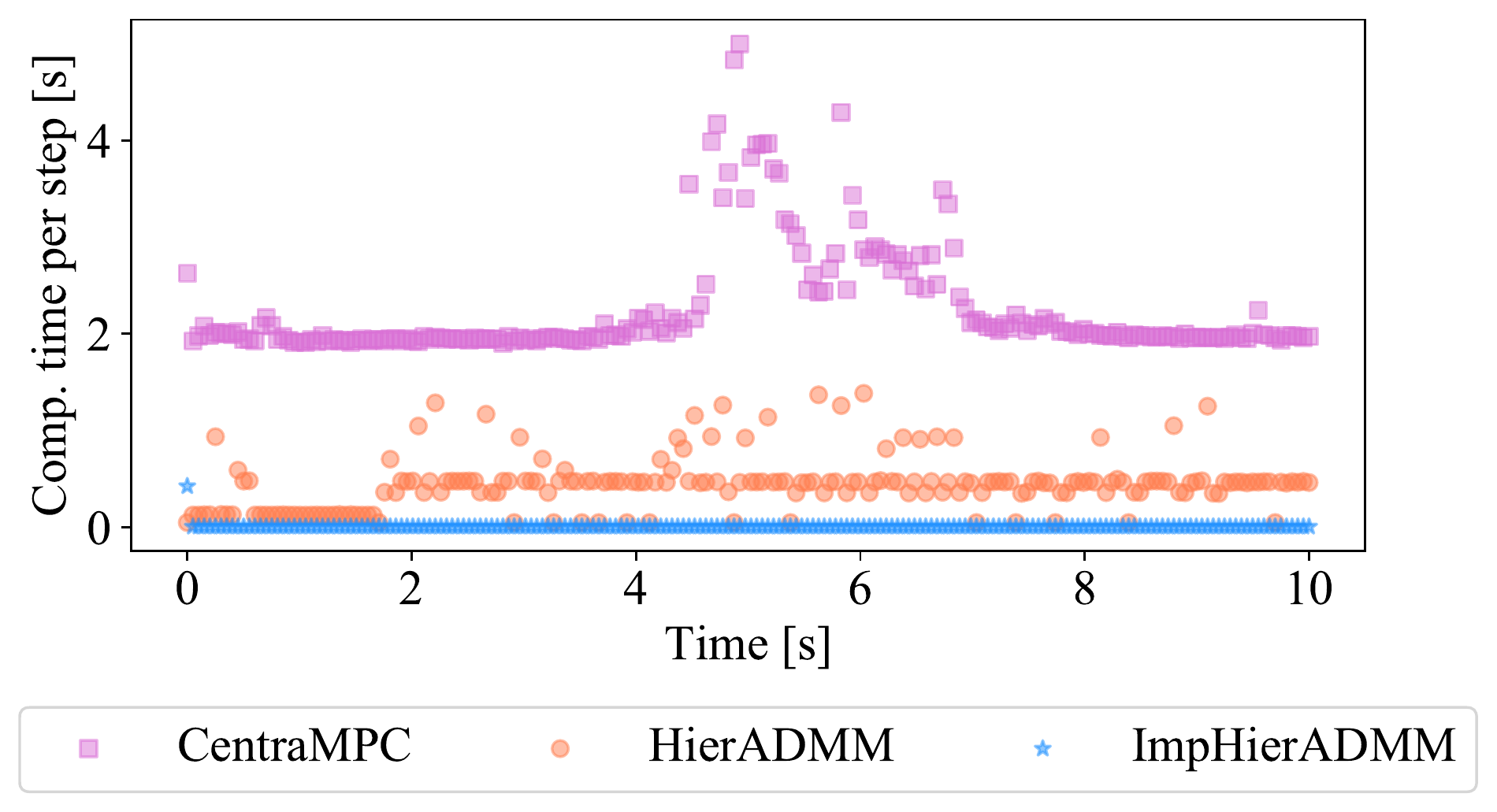}
	\caption{\label{fig:comptime_compare} Comparison of the computation time per step among the hierarchical ADMM, improved hierarchical ADMM, and centralized MPC solver solver.}
\end{figure}

\begin{figure}[!t]
	\centering
	\includegraphics[width=0.37\textwidth, trim=0 10 0 7,clip]{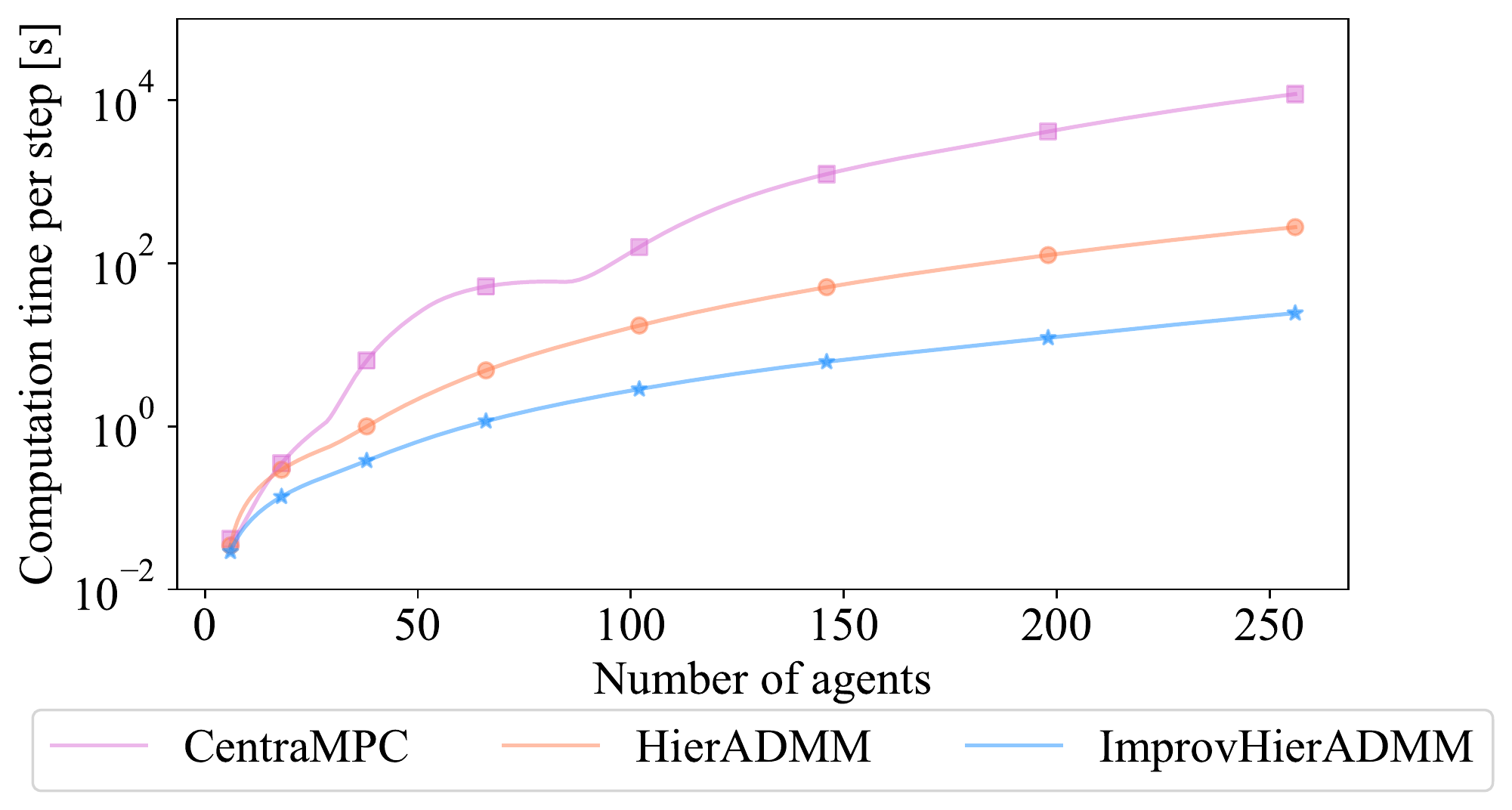}
	\caption{\label{fig:comptime_vs_numagents} Comparison of the computation time among the hierarchical ADMM, improved hierarchical ADMM, and centralized MPC solver with the increase of number of agents.}
\end{figure}

\section{Conclusion}
\label{section:conclusion_and_discussion}
In this paper, a large-scale nonconvex optimization problem for multi-agent decision making is investigated and solved by the hierarchical ADMM approach. For the multi-agent system, an appropriate DMPC problem is formulated and presented to handle the constraints and achieve the collective objective. Firstly, an innovation using appropriate slack variables is deployed, yielding a resulting transformed optimization problem, which is used to deal with the nonconvex characteristics of the problem. Then, the hierarchical ADMM is applied to solve the transformed problem in parallel. Next, the improved hierarchical ADMM is proposed to increase the computational efficiency and decrease the number of inner iterations. 
Additionally, it is analytically shown that the appropriate desired stationary point exists for the procedures of the hierarchical stages for algorithmic convergence. 
Finally, comparative simulations are conducted using a multi-UAV system as the test platform, which further illustrates the effectiveness and efficiency of the proposed methodology.

\ifCLASSOPTIONcaptionsoff
  \newpage
\fi

\bibliographystyle{IEEEtran}
\bibliography{reference}

\end{document}